   \documentclass[oneside]{amsart}
\usepackage{geometry}
\usepackage{microtype} 
\usepackage{tikz}
\usetikzlibrary{arrows,calc,matrix,topaths,positioning,scopes,shapes,decorations,decorations.markings} 
\usepackage{tkz-euclide}
\usetkzobj{all}

\usepackage{verbatim}
\usepackage{float}
\usepackage{textcomp}
\usepackage{amscd,amsfonts}
\usepackage{amsthm}
\usepackage{amstext}
\usepackage{amssymb}
\usepackage{xcolor}
\usepackage{graphicx}
\usepackage{esint}
\usepackage[unicode=true]{hyperref}

\usepackage{enumerate}
\usepackage{enumitem}
\newlist{enumD}{enumerate}{1}
\setlist[enumD]{label=(D\arabic*)}

\numberwithin{equation}{section}
\numberwithin{figure}{section}
\theoremstyle{plain}
\newtheorem{thm}{\protect\theoremname}
  \theoremstyle{definition}
  \newtheorem{defn}[thm]{\protect\definitionname}
  \theoremstyle{plain}
  \newtheorem{prop}[thm]{\protect\propositionname}
  \theoremstyle{remark}
  \newtheorem{rem}[thm]{\protect\remarkname}
  \theoremstyle{definition}
  \newtheorem{example}[thm]{\protect\examplename}
  \theoremstyle{plain}
  \newtheorem{lem}[thm]{\protect\lemmaname}
  \theoremstyle{plain}
  \newtheorem{cor}[thm]{\protect\corollaryname}

\def\g{{\mathfrak g}}

\usepackage{marginnote}
\reversemarginpar
\newcommand{\ie}{\emph{i.e.} }
\newcommand{\ot}{\otimes}
\newcommand{\cat}[1]{\mathbf{#1}}
\newcommand{\ca}[1]{\mathcal{#1}}
\def\co{\colon\thinspace}

\newcommand{\Tp}{\mathcal{PT}_{p}}
\newcommand{\K}{\mathbb{K}}

\newcommand{\TB}{\ca{PSB}}

\newcommand{\bsq}{\begin{tikzpicture}[scale=1,baseline=-2.5pt]
	\draw [black,fill=black] (-.06,-.06) rectangle (.06,.06); 
	\end{tikzpicture}}
\newcommand{\Prim}{\operatorname{Prim}}
\tikzstyle{my circle}=[draw, fill, circle, minimum size=3pt, inner sep=0pt]	
\tikzstyle{sq}=[draw, fill, rectangle, minimum size=3pt, inner sep=0pt]		
\newcommand{\lt}{\mathcal{LT}}
\newcommand{\lw}{\mathcal{LW}}
\newcommand{\w}{\mathcal{W}}

  \providecommand{\corollaryname}{Corollary}
  \providecommand{\definitionname}{Definition}
  \providecommand{\examplename}{Example}
  \providecommand{\lemmaname}{Lemma}
  \providecommand{\propositionname}{Proposition}
  \providecommand{\remarkname}{Remark}
\providecommand{\theoremname}{Theorem}

\begin{document}
	
	\title{Post-symmetric braces and integration of post-Lie algebras}
		\author{Igor Mencattini}
	\address{Universidade de S\~ao Paulo,
		Instituto de Ci\^encias Matem\'aticas e de Computa\c c\~ao,
		Avenida Trabalhador S\~ao-carlense, 400 -
		CEP: 13566-590 - S\~ao Carlos, SP -
		Brazil }
	\email{igorre@icmc.usp.br}
	\author{Alexandre Quesney}
	\address{Universidade de S\~ao Paulo,
		Instituto de Ci\^encias Matem\'aticas e de Computa\c c\~ao,
		Avenida Trabalhador S\~ao-carlense, 400 -
		CEP: 13566-590 - S\~ao Carlos, SP -
		Brazil }
	\email{math@quesney.org} 
	\author{Pryscilla Silva}
	\address{Universidade Estadual de Santa Cruz, Departamento de Ci\^encias Exatas, Campus Soane Nazar\'e de Andrade, Rodovia Jorge Amado, km 16, Bairro Salobrinho
		CEP 45662-900- Ilh\'eus,BA-Brasil}
	\email{psfsilva@uesc.br}



\date{\today}

	\begin{abstract}	
	The aim of this letter is twofold. Firstly, we introduce the post-Lie analogue of the notion of a symmetric brace algebra, termed in the sequel \emph{post-symmetric brace} algebra. These brace algebras are defined using a suitable algebraic operad, which turns out to be isomorphic to the  operad of the post-Lie algebras.  
	Secondly, using these new brace algebras, together with the so called \emph{post-Lie Magnus expansion}, we aim both to analyze the enveloping algebra of the corresponding post-Lie algebra and to compare the two Baker-Campbell-Hausdorff series there naturally defined. 
	\end{abstract}
	\maketitle

 \tableofcontents

\section{Introduction}

	Post-Lie algebras are non-associative algebras which later on appeared to play an important role in different areas of pure and applied mathematics. In this paper we are mainly concerned with some of their properties which are more easily detected through the operadic approach. To make the present work as self-contained as possible we summarize here below the necessary background, stressing differences and similarities with the better known class of the pre-Lie algebras.
	
	\emph{Pre}-Lie algebras have been introduced by Vinberg in \cite{Vinberg} in his studies about convex cones and, almost at the same time, they appeared in Gerstenhaber's foundational work \cite{Gersten} about the deformation theory of associative algebras. 
	Since then, pre-Lie algebras have been at the center of extensive investigations, especially  because of their importance in  combinatorics, mathematical physics, differential geometry, Lie theory and numerical analysis; see \cite{Manchon} and \cite{burde} for comprehensive reviews. %
	A pre-Lie algebra is a vector space $V$ endowed with a bilinear product $\triangleright:V\otimes V\rightarrow V$ whose associator 
	\begin{equation*}
	{\mathrm a}_{\triangleright}(x,y,z):=x\triangleright (y\triangleright z)-(x\triangleright y)\triangleright z,\,\forall x,y,z\in V,\label{eq:ass}
	\end{equation*}	
	satisfies 
		\begin{equation}
	{\mathrm a}_{\triangleright}(x,y,z)={\mathrm a}_{\triangleright}(y,x,z),\,\forall x,y,z\in V.\label{eq:preLie}
	\end{equation}
	The condition \eqref{eq:preLie} guaranties that $[-,-]:V\otimes V\rightarrow V$, defined by $[x,y]:=x\triangleright y-y\triangleright x$ for all $x,y\in V$, satisfies the Jacobi identity so that $V_{\text{Lie}}=(V,[-,-])$ is a Lie algebra.

	The universal enveloping algebra of a Lie algebra coming from a pre-Lie algebra $(V,\triangleright)$ was analyzed in depth in the papers \cite{GO1,GO2}, where it was shown that the existence of the pre-Lie product on $V$ provides the cofree symmetric coalgebra $S(V)$ with an associative but not commutative product
	$\ast:S(V)\otimes S(V)\rightarrow S(V)$, defined by
	\begin{equation}\label{eq: ast def}
		A\ast B = A_{(1)}(A_{(2)}\triangleright B),\,\forall A,B\in S(V).
	\end{equation}
	
	Endowed with such a product and its usual coalgebra structure, $S(V)$ becomes a bialgebra isomorphic to the universal enveloping algebra of $V_{\text{Lie}}$. 
	In Formula \eqref{eq: ast def} enters both the shuffle coproduct of $S(V)$ and a non-associative product $\triangleright:S(V)\otimes S(V)\rightarrow S(V)$ obtained extending to $S(V)$ the original pre-Lie product defined on $V$. Furthermore, the restriction $\triangleright: S(V)\otimes V\rightarrow V$ of the latter non-associative product defines a structure of \emph{symmetric brace} algebra on the vector space $V$, that is a family of operations   $B_{n}\in\operatorname{End}_{\mathbb K}(V^{\otimes n+1},V)$, $n\geq 1$, symmetric in the first $n$ variables and such that 
	\[
	B_n(v_1,...,v_n;u)=B_1(v_1;B_{n-1}(v_2,...,v_n;u))-\sum_{k=2}^n B_{n-1}(v_2,...,B_1(v_1;v_k),...,v_n;u)
	\]
	for all $v_i$'s and $u$ in $V$.  
	In \cite{GO1,LM} and \cite{CL} it was proved that the symmetric brace algebras form a category isomorphic to the category of the pre-Lie algebras.
	More precisely, in \cite{CL} was proposed a presentation, in terms of the set of {\em non-planar} rooted trees, of the operad  $\ca{P}re\ca{L}ie$ controlling the pre-Lie algebras. As it turned out that the algebras over this operad are the symmetric brace algebras, this operadic approach provided a convenient way to prove the above mentioned isomorphism between the category of symmetric brace algebras and the category of pre-Lie algebras.  

\emph{Post}-Lie algebras were introduced by Vallette in \cite{Val} as being the algebras over the operad $\ca{P}ost\ca{L}ie$, which is the Koszul dual of the operad controlling the commutative trialgebras. 
   They were further analyzed a few years later by Munthe-Kaas and Ludervold in \cite{MKL}, in their study of the {\em order conditions} for the Lie group integrators. 
   Formally speaking a post-Lie algebra is a Lie algebra $(\mathfrak g, [-,-])$ endowed with non- associative product $\triangleright:\mathfrak g\otimes\mathfrak g\rightarrow\mathfrak g$ that satisfies	
	\begin{align}
	z\triangleright [x,y]&=[z\triangleright x,y]+[x,z\triangleright y]\label{eq:pl1} &\text{ and } \\
	[x,y]\triangleright z&={\mathrm a}_{\triangleright}(x,y,z)-{\mathrm a}_{\triangleright}(y,x,z) &\text{ for all } x,y,z\in V. \label{eq:pl2}
	\end{align}
	It is worth to note that, in spite the post-Lie product does not yield a Lie bracket by antisymmetrization, the bilinear product $[[-,-]]:\g\otimes\g\rightarrow\g$, defined by
	\begin{equation}
	[[x,y]]=x\triangleright y-y\triangleright x+[x,y],~\forall x,y\in\g,\label{eq:newLie}
	\end{equation}
	defines on $\g$ another Lie algebra structure, which, from now on, will be denoted by $\overline\g=(\g,[[-,-]])$. 
	Note also that \eqref{eq:pl2} implies that a post-Lie algebra with a trivial Lie bracket is pre-Lie algebra. 
	
	After their introduction post-Lie algebras have been extensively studied, since they appeared in several different areas of mathematics like Lie theory (see \cite{Burde1, Burde2, Burde3, KLM}),  mathematical physics (see \cite{BGN, EFLMMK, KI}) and numerical analysis (see \cite{MKL, CEFMK, CEFO, MuntheKaas}). 
	Although the post-Lie algebras were introduced to the realm of numerical analysis in \cite{MKL}, their relevance for the theory of numerical integration could be  traced back to the fundamental work \cite{MW}, where the authors performed a deep investigation of the Hopf algebra relevant for the algebraic description of the Lie group integrators.
	More precisely, in \cite{MW} was introduced a cocommutative but non commutative Hopf algebra whose underlying linear space is generated by the forests  of {\em planar} rooted trees (there called {\em ordered} rooted trees). This algebraic structure was then enhanced with an operation of {\em left-grafting} and was termed {\em $D$--algebra}. 
	With this concept at hand, the authors of \cite{MW} could inscribe the theory of Lie group integrators in the algebraic framework of the theory of Hopf algebras. 
	Soon after it was understood that $D$--algebras played the role of universal enveloping algebras for post-Lie algebras. More precisely, in \cite{MKL} it was shown that the universal enveloping algebra functor 
	\begin{equation}\label{eq: functor env. postLie1}
	\mathcal{U} \co \cat{PostLie}  \to \cat{D-algebra}  
	\end{equation}  
	provides an adjunction  between the category of post-Lie algebras and the category of $D$--algebras, whose adjoint is the \emph{derivation} functor (see \cite[Definition 3.5]{MKL}).\\  
	The properties of \eqref{eq: functor env. postLie1} were further investigated in \cite{KLM}, where it was shown that for any post-Lie algebra $(\g,\triangleright)$ there exists a \emph{unique} extension $\triangleright:\mathcal U(\g)\otimes\mathcal U(\g)\rightarrow\mathcal U(\g)$ of the post-Lie product which satisfies suitable compatibility conditions with its canonical coalgebra structure, see \cite[Proposition 3.1]{KLM}. Furthermore, it was proven that the product $\ast:\mathcal U(\g)\otimes\mathcal U(\g)\rightarrow\mathcal U(\g)$, defined, in analogy to \eqref{eq: ast def}, by 
	\begin{equation}\label{eq: ast def2}
	A\ast B=A_{(1)}(A_{(2)}\triangleright B), ~\text{ for all $A$ and $B$ in }\mathcal U(\g), 
	\end{equation} 
	makes $\mathcal U(\g)$ into a bialgebra isomorphic to $\mathcal U(\overline\g)$. It is worth to mention that the product $\ast$ corresponds to the \emph{Grossman-Larson} product on the free post-Lie algebra of planar rooted trees; see \cite{MW} and also \cite{MKL,KLM}. 
\\
	
\emph{Main results of the present work. } 
The main aim of the present paper is to introduce the post-Lie analogue of the symmetric brace algebras and to study some of its properties, emphasizing its relations with other mathematical structures naturally associated to every post-Lie algebra.  

First recall that a relevant point in the construction of the operad $\ca{RT}$ in \cite{CL} is  that the symmetric braces (which are nothing but iterations of the pre-Lie product) are naturally encoded in the combinatorics of \emph{non-planar} rooted trees, which therefore  provide a very convenient presentation of the free pre-Lie algebra. 
In analogy with this case, in Section \ref{sec: (sub) PL via PRT } it is shown that the presentation of the free post-Lie algebra via \emph{planar} rooted trees, see \cite{MW,MKL,KLM}, also comes from an operadic structure. The operad in question, denoted by $\TB$, is shown to be isomorphic to $\ca{P}ost\ca{L}ie$ in Theorem \ref{thm: isom}.  
 
 The $\TB$--algebras are characterized in Section \ref{sec:tbalgebra}, where it is shown that they are post-Lie algebras endowed with brace type operations, see Proposition \ref{prop:tb}. 
 We term them \emph{post-symmetric brace algebras} and we propose them as our candidate to be the desired post-Lie analogue of the symmetric brace algebras.  
The post-symmetric brace algebras turn out to be a convenient tool to further investigate the properties of the product defined in \eqref{eq: ast def2}. 
More in details, in Section \ref{ssec:tbvsda}, we introduce the notion of $D$--\emph{bi}algebras and we show that \eqref{eq: functor env. postLie1}, seen as a functor 
\begin{equation}\label{eq: functor env. postLie}
\mathcal{U} \co \cat{PostLie}  \to \cat{D-bialgebra},  
\end{equation}
is full and faithful (see Theorem \ref{thm:tbrD}), and provides an adjunction of categories (see Proposition \ref{prop: adjoint})  whose adjoint is the \emph{primitive elements} functor,
\begin{equation}\label{eq: functor primitive}
\operatorname{Prim}\co \cat{D-bialgebra}\to\cat{PostLie},
\end{equation}  
which associates to every $D$--bialgebra its primitive elements, see Definition \ref{def:dbialgebra}.
 The faithfulness of the functor $\mathcal{U}$ follows, essentially, from the observation that the post-symmetric braces on a Lie algebra $\g$ encode the left part of the $D$--product on $\mathcal{U}(\g)$.

In Section \ref{sec:magnus}, we analyze the so called \emph{post-Lie Magnus expansion}, introduced in \cite{EFLMMK}, see also \cite{KLM,KI}, from the view-point of the $\TB$--algebras. In more details, in Proposition \ref{prop:bchprop} we relate the Baker-Campbell-Hausdorff formulas of the Lie algebras $(\g,[-,-])$ and $\overline{\g}$ to the post-Lie Magnus expansion. 
This generalizes a result  proven for the first time in \cite{AG} about the integration of a pre-Lie algebra in terms of the group of the so called \emph{formal flows}. We recover such a result in Corollary \ref{cor:ag}, as a consequence of our Proposition \ref{prop:bchprop}.

\subsection{Conventions}\label{sec: convention operads} 

Throughout the paper $\mathbb K$ will denote a field of characteristic zero.  
Tensor product will be taken over $\mathbb K$, and the tensor product of two $\mathbb K$--vector spaces $V$ and $W$ will be denoted by $V\ot W$.  

For a $\mathbb K$--vector space $V$, we let $T(V)=\oplus_{n\geq 0}V^{\otimes n}$, where $V^{\otimes 0}=\mathbb K$, be its free tensor algebra.  
The map $\Delta_{sh}\co V\to T(V)\ot T(V)$ given by $\Delta_{sh}(x)=x\ot 1 + 1\ot x$ for all $x\in V$ extends uniquely to $T(V)$  as a morphism of algebras. The resulting map $\Delta_{sh}\co T(V)\rightarrow T(V)\otimes T(V)$ is called \emph{shuffle} coproduct and  makes the free tensor algebra into a cocommutative bialgebra.  
It is described as follows.  
Let $X=x_1\ot \cdots \ot x_n$ be in $V^{\ot n}$. For a sub order set $I=\{i_{1}<\cdots <i_{k}\}$ of $\{1<...<n\}$, we let $X_I$ denote the element $x_{i_1}\ot \cdots \ot x_{i_k}$ in $V^{\ot |I|}$. One has  
\[
\Delta_{sh} (X)=\sum_{I\coprod J}X_{I}\otimes X_{J},
\] 
where the sum runs over the ordered partitions $I\coprod J$ of $\{1<...<n\}$. 

In general, for a coproduct $\Delta\co W\to W\ot W$,  we will use the \emph{Sweedler} notation in its compact form: 
\[
\Delta (X)=X_{(1)}\otimes X_{(2)} ~~\text{ for all } X\in W. 
\] 

To save notation, for a linear map $f\co V^{\ot k} \to W$, we will  write $f(x_1,...,x_k)$ for $f(x_1\ot ...\ot x_k)$.

We will consider operads in the category of $\mathbb K$--vector spaces. Our convention follows \cite{LV} to which we refer for more details. 
In brief, an operad  $\ca{O}$ is an $\mathbb{S}$--module in the category of the $\mathbb K$--vector spaces, together with partial compositions $\circ_i\co \ca{O}(m) \ot \ca{O}(n) \to \ca{O}(m+n-1)$ for $1\leq i \leq m$ that satisfy associativity, equivariance and unit axioms. 
In particular, this means that for each $n\geq1$, the vector space $\ca{O}(n)$ is acted on by the symmetric group $\mathbb{S}_n$, and that the maps $\circ_i$ are equivariant for this action.  
For instance, for a vector space $V$, the $\mathbb{S}$--module $\operatorname{End}_V$ given by $\operatorname{End}_V(n)=\operatorname{Hom}_{\mathbb K}(V^{\ot n},V)$ for $n\geq1$, is an operad for the partial composition of linear maps: 
\begin{equation*}
f\circ_i g (x_1, ... , x_{n+m-1}) := f(x_1, ..., x_{i-1}, g(x_i, ..., x_{i+n-1}), x_{i+n},..., x_{n+m-1}),
\end{equation*} 
for every two maps $f\in \operatorname{End}_V(m)$  and $g\in \operatorname{End}_V(n)$, and  all $x_j\in V$. The action of the symmetric group is given by permutation of the variables. 
A vector space $V$ is called an  $\ca{O}$--\emph{algebra} if there is a morphism of operads $\ca{O}\to \operatorname{End}_V$. 

An ideal $\ca{I}$ of an operad $\ca{O}$ is a sub $\mathbb{S}$--module of $\ca{O}$ such that the maps $\circ_i$ co-restrict to $\ca{I}$ whenever they are restricted to $\ca{I}$ in any of its two components \ie they induce maps $\circ_i\co \ca{I}(m) \ot \ca{O}(n) \to \ca{I}(m+n-1)$ and $\circ_i\co \ca{O}(m) \ot \ca{I}(n) \to \ca{I}(m+n-1)$. 
In particular, the $\mathbb{S}$--module quotient $(\ca{O}/\ca{I})(n):=\ca{O}(n)/\ca{I}(n)$, for $n\geq 1$, has a structure of operad.

To end this preamble, let us make a comment on decompositions. 
Let  $\ca{O}$ be an operad and $T$ an element of $\ca{O}(n)$. 
Consider a decomposition of $T$, say 
$T=(\cdots(\cdots ((S_1 \circ_{i_1} S_2 ) \circ_{i_2} S_3) \cdots S_{k-1} )\circ_{i_{k-1}} S_k)^{\sigma}$. 
Suppose that for each $n$, the action of $\mathbb{S}_n$ on $\ca{O}(n)$ is free. 
Because of the $\mathbb{S}$--equivariance of the maps $\circ_i$, by applying equivariant actions of $\mathbb{S}$ one gets other decompositions of $T$. If there is no additional decomposition of $T$, anyone of the above decompositions is called  \emph{unique $\mathbb{S}$--equivariant}.

\section{An operad of planar trees}\label{sec: Operad of PRT}	

In this preliminary section, we set up notation and introduce the relevant combinatorial objects for our needs. 
We also define an operad of trees, which will be the starting point for defining other operads important for the present work.

\begin{defn} 
	A \emph{planar rooted tree} is an isomorphism class of contractible graphs, embedded in the plane, and endowed with a distinguished vertex, called the \emph{root}, to which is attached an adjacent half-edge, called the \emph{root-edge} of the planar tree. 
\end{defn}

For a planar rooted tree $T$, we let $V(T)$ be the set of all its vertices. On it, we consider two orders: 
\begin{itemize}
	\item  The \emph{level partial} order $\prec$ defined by orienting the edges of $T$ towards the root, except the root-edge. 
	For two  vertices $u$ and $v$ of $ V(T)$, we write $v \prec u$ if there is a string of oriented edges from $v$ to $u$. In particular, the root is  maximal for this partial order. 
	\item The \emph{canonical linear} order $<$: starting from the root-edge of $T$, we run along $T$ in the clockwise direction, passing trough each edge once per direction. The order we meet the vertices for the first time gives the order $<$; see Example \ref{ex: angles}. In particular the root is the minimal element for $<$. 
\end{itemize}

Pictorially, our trees are drawn with the root at the bottom, and the order on the set of the incoming edges of a vertex is given by the clockwise direction, i.e. from the left to the right.

For each vertex $v\in V(T)$, consider  a small disc centered at $v$. The \emph{angles} of $v$ are the connected components of the disc, when removed its intersection with the tree. The angles are ordered in the clockwise direction, starting from the component in between the outgoing edge of $v$ and its left most incoming edge (if it exists). 
We let $Ang^{min}(T)$ be the set of the \emph{minimal} angles of all the vertices of $T$. It is in canonical bijection with the set $V(T)$, which endows it with an order; see Example \ref{ex: angles}.

\begin{example}\label{ex: angles} Let $T$ be the following planar rooted tree, together with its angles:  
	\begin{equation*} 
	\pgfdeclarelayer{background}
	\pgfsetlayers{background,main} 
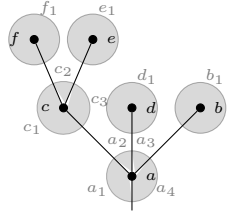
	\begin{tikzpicture}
	[  
	level distance=0.7cm, 
	level 2/.style={sibling distance=0.6cm}, sibling distance=0.7cm,baseline=1.5ex,scale=1.3]
	\node (1) [my circle,label=right:\tiny{$a$}] {} [grow=up]
	{
		child {node (2) [my circle,label=right:\tiny{$b$}]  {}} 
		child {node (4) [my circle,label=right:\tiny{$d$}] {}} 
		child {node (3) [my circle,label=left:\tiny{$c$}] {}
			child {node (5) [my circle,label=right:\tiny{$e$}] {}}
			child {node (6) [my circle,label=left:\tiny{$f$}] {}} }
	};
	\coordinate (Root) at (0,-.25);
	\draw [-] (0,-.35) -- (0,0) ;
	\begin{pgfonlayer}{background}
	\tkzMarkAngle[fill=gray,size=0.26cm,opacity=.3](3,1,Root)
	\tkzLabelAngle[pos=-0.45](3,1,Root){\textcolor{gray}{\tiny{$a_1$}}}
	\tkzMarkAngle[fill=gray,size=0.26cm,opacity=.3,draw opacity=.3](4,1,3)
	\tkzLabelAngle[pos=0.35](4,1,3){\textcolor{gray}{\tiny{$a_2$}}}
	\tkzMarkAngle[fill=gray,size=0.26cm,opacity=.3,draw opacity=.3](2,1,4)
	\tkzLabelAngle[pos=0.35](2,1,4){\textcolor{gray}{\tiny{$a_3$}}}
	\tkzMarkAngle[fill=gray,size=0.26cm,opacity=.3,draw opacity=.3](Root,1,2)
	\tkzLabelAngle[pos=0.35](Root,1,2){\textcolor{gray}{\tiny{$a_4$}}}
	\tkzMarkAngle[fill=gray,size=0.27cm,opacity=.3](6,3,1)
	\tkzLabelAngle[pos=-0.45](6,3,1){\textcolor{gray}{\tiny{$c_1$}}}
	\tkzMarkAngle[fill=gray,size=0.27cm,opacity=.3,draw opacity=.3](5,3,6)
	\tkzLabelAngle[pos=0.35](5,3,6){\textcolor{gray}{\tiny{$c_2$}}}
	\tkzMarkAngle[fill=gray,size=0.27cm,opacity=.3,draw opacity=.3](1,3,5)
	\tkzLabelAngle[pos=0.35](1,3,5){\textcolor{gray}{\tiny{$c_3$}}}
	\draw[fill=gray,opacity=.3] (6)  circle (.26)   ;
	\draw (6) node [above,yshift=.2cm,xshift=.2cm] {\textcolor{gray}{\tiny{$f_1$}}}; 
	\draw[fill=gray,opacity=.3] (5)  circle (.26)   ;
	\draw (5) node [above,yshift=.2cm,xshift=.2cm] {\textcolor{gray}{\tiny{$e_1$}}}; 
	\draw[fill=gray,opacity=.3] (4)  circle (.26)   ;
	\draw (4) node [above,yshift=.2cm,xshift=.2cm] {\textcolor{gray}{\tiny{$d_1$}}}; 
	\draw[fill=gray,opacity=.3] (2)  circle (.26)   ;
	\draw (2) node [above,yshift=.2cm,xshift=.2cm] {\textcolor{gray}{\tiny{$b_1$}}}; 
	\end{pgfonlayer}
	\end{tikzpicture}
	\end{equation*}
	The ordered set of the vertices of $T$ is $V(T)=\{a<c<f<e<d<b\}$. 
	The ordered set of angles of the vertex $c$ is $\{c_1<c_2<c_3\}$, the ordered set of the minimal angles of $T$ is  $Ang^{min}(T)= \{a_1<c_1<f_1<e_1<d_1<b_1\}$. 
\end{example}

From now on, when there is no ambiguity, planar rooted trees are simply called trees.  
We will consider trees with labelings, or more in general, with partial labelings.

\begin{defn} 
	Let $T$ be a tree and let $U$ be a subset of  $V(T)$.  
	A \emph{$U$--label} of $T$ is a bijection $l\co U \rightarrow \{1,...,n\}$. 
	A tree $T$ equipped with a $U$--label is called \emph{partially labeled}; it is called \emph{fully labeled} (or simply, labeled) if $U= V(T)$. 
\end{defn}

Let us denote by  $l_{id}\co U \rightarrow \{1<...<n\}$ the unique isomorphism of ordered sets, where $U$ has the induced order from  $V(T)$.  
Since $\mathbb{S}_n$ acts freely and transitively on the set of $U$--labels by post-composition,  
one can write a $U$--label in a unique form by $l_{\sigma}\co U \rightarrow \{1,...,n\}$ for some $\sigma\in \mathbb{S}_n$ (and $l_{\sigma}:=\sigma l_{id}$).

\begin{example}
	For $(312) \in \mathbb{S}_3$ we have the following examples of partially labeled trees: \\
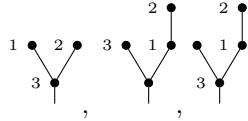
	\begin{equation*}
	\begin{tikzpicture}
	[baseline, my circle/.style={draw, fill, circle, minimum size=3pt, inner sep=0pt},  
	level 2/.style={sibling distance=0.6cm,level distance=0.5cm}, sibling distance=0.5cm, level 1/.style={level distance=0.4cm}]
	\node {} [grow=up]
	{child{node [my circle,label=left:\tiny{$3$}] {} child {node (A) [my circle,label=left:\tiny{$2$}]  {}
			} child {node [my circle,label=left:\tiny{$1$}] {}} }};
	\end{tikzpicture}, 
	\begin{tikzpicture}
	[baseline, my circle/.style={draw, fill, circle, minimum size=3pt, inner sep=0pt},  
	level 2/.style={sibling distance=0.6cm,level distance=0.5cm}, sibling distance=0.5cm, level 1/.style={level distance=0.4cm}]
	\node {} [grow=up]
	{child{node [my circle] {} child {node (A) [my circle,label=left:\tiny{$1$}] {} child {node [my circle,label=left:\tiny{$2$}] {}}}
			child {node  [my circle,label=left:\tiny{$3$}]  {}
	}  }};
	\end{tikzpicture}, 
	\begin{tikzpicture}
	[baseline, my circle/.style={draw, fill, circle, minimum size=3pt, inner sep=0pt},  
	level 2/.style={sibling distance=0.6cm,level distance=0.5cm}, sibling distance=0.5cm, level 1/.style={level distance=0.4cm}]
	\node {} [grow=up]
	{child{node [my circle,label=left:\tiny{$3$}] {} child {node (A) [my circle, label=left:\tiny{$1$}] {} child {node [my circle,label=left:\tiny{$2$}] {}}}
			child {node  [my circle]  {}
	}  }};
	\end{tikzpicture}.  
	\end{equation*}
\end{example}

\begin{defn}
	For a partially labeled tree $T$, let $Ang_{lab}^{min}(T)$ be the set of those minimal angles of the labeled vertices of $T$. 
\end{defn}
\begin{defn}
	Recall that our trees have oriented edges. For $1\leq i \leq n$, let $In(i)$ be the order set of \emph{incoming} edges of the vertex labeled by $i$.  The order on $In(i)$ is the one induced by the clockwise orientation, that is, from the left to the right.  
\end{defn}

Let us define an operadic structure on the partially labeled trees.  
Let $\Tp(n)$ be the $\K$--vector space generated by the set of partially labeled trees, whose $U$--labels are with $U$ of cardinality $n$. 
For $m,n\geq 1$ and  $1\leq i \leq m$, let  
\begin{equation}\label{eq: partial compo partial planar tree}
\circ_i\co \Tp(m)\ot \Tp(n) \to \Tp(m+n-1)  
\end{equation}
be as follows. 

  For each $T\in \Tp(m) $, $R\in \Tp(n)$ and for a map of sets 
 \[
 \phi\co In(i)\to Ang_{lab}^{min}(R),
 \] 
 define $T\circ_i^{\phi} R$ to be the tree obtained as follows: substitute the vertex labeled by $i$ by the tree $R$, and graft the incoming edges  of $i$ to the labeled vertices of $R$ following the map $\phi$. 
 The grafting is required to be performed in such a way that it respects the natural order of each fiber of $\phi$. More precisely, since for every $\alpha$ the fiber $\phi^{-1}(\alpha)$ is endowed with the induced order from $In(i)$ (given by $Ang^{min }(T)$), when performing the substitution, the incoming edges of $\phi^{-1}(\alpha)$ preserves this order. 
 The outgoing edge of $i$  is identified with the root-edge of $R$.  
 The tree $T\circ_i^{\phi} R$ has $m-1+n$ labeled vertices, and the partial labeling is given by classical re-indexation.  
 Explicitly, for a $U$--label $l^T$ of $T$  and a $U'$--label $l^R$ of $R$, the resulting 
	label of $T\circ_i^{\phi} R$, say $l''$, is as follows. 
	Suppose $v$ be the vertex of $T$ labeled by $i$ and consider the canonical inclusion of $U\setminus \{v\}$ into the set of the vertices of $T\circ_i^{\phi} R$, and similarly for $U'$. 
	The $U\setminus \{v\}\cup U'$--label $l''$ is defined by
 \begin{equation*}
 l''(w)= 
 \begin{cases}
 	l^T(w) & \text{ if }  w\in U\setminus \{v\} \text{ and } 1 \leq l^T(w) \leq i-1 \\
 	l^R(w)+i-1 & \text{ if } w\in U'\\ 
  	l^T(w)+n & \text{ if }  w\in U\setminus \{v\} \text{ and } l^T(w) \geq i+1.
 	 	\end{cases}
 \end{equation*} 
 For two trees $T$ and  $R$ as above, their partial composition is given  by \begin{equation}\label{eq: explicit partial compo partial planar tree}
 T\circ_i R = \sum_{\phi\co In(i) \to Ang_{lab}^{min}(R)} T\circ_i^{\phi} R.
 \end{equation} 
 We extend this by linearity to get \eqref{eq: partial compo partial planar tree}. For instance, one has 
 \begin{equation*}
 \begin{tikzpicture}
 [baseline, my circle/.style={draw, fill, circle, minimum size=3pt, inner sep=0pt}, level distance=0.5cm, 
 level 2/.style={sibling distance=0.6cm}, sibling distance=0.6cm,baseline=1.5ex]
 \node [my circle,label=left:\tiny{$1$}] {} [grow=up]
 {
 	child {node [my circle,label=left:\tiny{$3$}]  {}} 
 	child {node [my circle,label=left:\tiny{$2$}] {}} 
 };
 \draw [-] (0,-.25) -- (0,0) ;
 \end{tikzpicture}
 \circ_1 
 \begin{tikzpicture}
 [baseline, my circle/.style={draw, fill, circle, minimum size=3pt, inner sep=0pt}, level distance=0.5cm, 
 level 2/.style={sibling distance=0.6cm}, sibling distance=0.6cm,baseline=1.5ex]
 \node [my circle,label=left:\tiny{$1$}] {} [grow=up]
 {
 	child {node [my circle,label=left:\tiny{}]  {}} 
 };
 \draw [-] (0,-.25) -- (0,0) ;
 \end{tikzpicture}
 =
 \begin{tikzpicture}
 [baseline, my circle/.style={draw, fill, circle, minimum size=3pt, inner sep=0pt}, level distance=0.5cm, 
 level 2/.style={sibling distance=0.6cm}, sibling distance=0.6cm,baseline=1.5ex]
 \node [my circle,label=left:\tiny{$1$}] {} [grow=up]
 {
 	child {node [my circle,label=left:\tiny{}]  {}} 
 	child {node [my circle,label=left:\tiny{$3$}] {}} 
 	child {node [my circle,label=left:\tiny{$2$}] {}}
 };
 \draw [-] (0,-.25) -- (0,0) ;
 \end{tikzpicture}
 \end{equation*}
 and 
 \begin{equation}\label{eq: exple compo prt}
 \begin{tikzpicture}
 [baseline, my circle/.style={draw, fill, circle, minimum size=3pt, inner sep=0pt}, level distance=0.5cm, 
 level 2/.style={sibling distance=0.6cm}, sibling distance=0.6cm,baseline=1.5ex]
 \node [my circle,label=left:\tiny{$1$}] {} [grow=up]
 {
 	child {node [my circle,label=left:\tiny{$3$}]  {}} 
 	child {node [my circle,label=left:\tiny{$2$}] {}} 
 };
 \draw [-] (0,-.25) -- (0,0) ;
 \end{tikzpicture}
 \circ_1 
 \begin{tikzpicture}
 [baseline, my circle/.style={draw, fill, circle, minimum size=3pt, inner sep=0pt}, level distance=0.5cm, 
 level 2/.style={sibling distance=0.6cm}, sibling distance=0.6cm,baseline=1.5ex]
 \node [my circle,label=left:\tiny{$1$}] {} [grow=up]
 {
 	child {node [my circle,label=left:\tiny{$2$}]  {}} 
 };
 \draw [-] (0,-.25) -- (0,0) ;
 \end{tikzpicture}
 =
 \begin{tikzpicture}
 [baseline, my circle/.style={draw, fill, circle, minimum size=3pt, inner sep=0pt}, level distance=0.5cm, 
 level 2/.style={sibling distance=0.6cm}, sibling distance=0.6cm,baseline=1.5ex]
 \node [my circle,label=left:\tiny{$1$}] {} [grow=up]
 {
 	child {node [my circle,label=left:\tiny{$2$}]  {}} 
 	child {node [my circle,label=left:\tiny{$4$}] {}} 
 	child {node [my circle,label=left:\tiny{$3$}] {}}
 };
 \draw [-] (0,-.25) -- (0,0) ;
 \end{tikzpicture}
 +
 \begin{tikzpicture}
 [baseline, my circle/.style={draw, fill, circle, minimum size=3pt, inner sep=0pt}, level distance=0.5cm, 
 level 2/.style={sibling distance=0.6cm}, sibling distance=0.6cm,baseline=1.5ex]
 \node [my circle,label=left:\tiny{$1$}] {} [grow=up]
 {
 	child {node [my circle,label=left:\tiny{$2$}]  {} 
 		child {node [my circle,label=left:\tiny{$4$}] {}} 
 	}
 	child {node [my circle,label=left:\tiny{$3$}] {} }
 };
 \draw [-] (0,-.25) -- (0,0) ;
 \end{tikzpicture}
 +
 \begin{tikzpicture}
 [baseline, my circle/.style={draw, fill, circle, minimum size=3pt, inner sep=0pt}, level distance=0.5cm, 
 level 2/.style={sibling distance=0.6cm}, sibling distance=0.6cm,baseline=1.5ex]
 \node [my circle,label=left:\tiny{$1$}] {} [grow=up]
 {
 	child {node [my circle,label=left:\tiny{$2$}]  {} 
 		child {node [my circle,label=left:\tiny{$3$}] {}} 
 	}
 	child {node [my circle,label=left:\tiny{$4$}] {} }
 };
 \draw [-] (0,-.25) -- (0,0) ;
 \end{tikzpicture}
 +
 \begin{tikzpicture}
 [baseline, my circle/.style={draw, fill, circle, minimum size=3pt, inner sep=0pt}, level distance=0.5cm, 
 level 2/.style={sibling distance=0.6cm}, sibling distance=0.6cm,baseline=1.5ex]
 \node [my circle,label=left:\tiny{$1$}] {} [grow=up]
 {
 	child {node [my circle,label=left:\tiny{$2$}]  {} 
 		child {node [my circle,label=left:\tiny{$4$}] {}} 
 		child {node [my circle,label=left:\tiny{$3$}] {}}
 	}
 };
 \draw [-] (0,-.25) -- (0,0) ;
 \end{tikzpicture}
 .
 \end{equation}

\begin{prop}\label{pro:operadtp}
The $\mathbb S$-module $\Tp=\oplus_{n\geq 1}\Tp(n)$ endowed with the partial composition maps defined in \eqref{eq: partial compo partial planar tree} becomes a symmetric operad. 
\end{prop}
\begin{proof}
	This is routine check. 
\end{proof}

\section{The operad $\TB$}\label{sec: (sub) PL via PRT }

In this section we introduce the operad $\TB$ defined in third author's PhD thesis \cite{PryscillaThesis}. It can be seen as the post-Lie analogue of the operad of rooted trees defined in \cite{CL}. 
Explicitly, we show that $\TB$ is isomorphic to the operad of the post-Lie algebras.

\subsection{Definition of $\TB$}

Let $\mathcal L=\oplus_{n\geq 1}\mathcal L(n)$ be the suboperad of  $\Tp$ generated by the fully labeled trees.

For each $n\geq 2$ let $\mathcal{W} (n) \subset \mathcal{PT}_p (n)$ be the $\mathbb{K}$--subvector space generated by those partially labeled trees $T$ that  satisfy: 
\begin{enumerate}
	\item[(a)] the root of $T$ is unlabeled; 
	\item[(b)] if a vertex of $T$ is unlabeled, then so is its $\prec$--successor;  
	\item[(c)] each unlabeled vertex of $T$ has exactly two incoming edges. 
\end{enumerate}

We let
\begin{equation*}
\ca{LW}(1):=\ca{L}(1) \text{ and }  \ca{LW}(n):=\mathcal L(n)\oplus\mathcal W(n) \text{ for } n\geq 2. 
\end{equation*} 
We will prove that $\ca{LW}=\oplus_{n\geq 1}\ca{LW}(n)$ carries a natural operad structure.

\subsubsection{Vertex-wise action of the symmetric group} \label{sec: symm action vertex}

Let $R$ be a tree in $\ca{W}(m)$ and let $V_{unl}(R)$  denote the set of its unlabeled vertices. 
For  $v\in V_{unl}(R)$, let $R^v$ be the maximal subtree of $R$ with root $v$. 
Since $v$ has exactly two incoming edges, $R^v$ can be written as a corolla $C(v;R^v_1,R^v_2)$, where  $\{R^v_1<R^v_2\}$  is the ordered set of the maximal subtrees of $R^v$ that does not contain $v$; the order is induced by the canonical order $<$ on $R$. 

To any permutation $\sigma \in {\mathbb S}_2$ one may associate the tree $R_{\sigma}$ that is obtained from $R$  by changing the subtree $R^v$ into  
\begin{equation*}
R^v_{\sigma}:=C(v;R^v_{\sigma(1)},R^v_{\sigma(2)}).
\end{equation*}
The labeling is unchanged: if a vertex of $R^v$ has a label, then its image in $R^v_{\sigma}$ has the same label. 

More generally, to any tuple of permutations $\sigma \in \mathbb{S}_{2}^{\times |V_{unl}(R)|}$ one may associate a tree $R_{\sigma}$ obtained by applying the above construction to each vertex $v$ of $V_{unl}(R)$. 
Since the order we perform the iteration does not matter, this is well defined. 
 
For a tuple $\sigma=(\sigma_1,...,\sigma_s) \in \mathbb{S}_{2}^{\times |V_{unl}(R)|}$  we set $\epsilon(\sigma)=(-1)^{sgn(\sigma_1)+...+sgn(\sigma_s)}$, where $sgn(\sigma_i) $ is  the signature of the permutation $\sigma_i\in {\mathbb S}_{2}$.

\subsubsection{Contracting trees}

Let $T$ be a partially labeled  tree with an unlabeled vertex $v$ that is not the root, and let $v^+$ be its $\prec$--successor. 
Denote by  $Con_v(T)$ the partially labeled tree obtained from $T$ by contracting the edge linking $v$ and $v^+$. 
The resulting vertex, which comes from the merging of $v$ and $v^+$, is endowed with the same label than the vertex $v^+$ of $T$, if any.
\\

For a tree $R\in \ca{W}(n)$, we denote by $Con(R)$ the tree in $\Tp(n)$ obtained from $R$ by contracting all the edges that are bounded by two unlabeled vertices. Thus, in $Con(R)$, the root is the unique unlabeled vertex.    
For instance, 
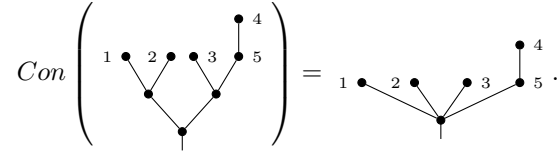
\begin{equation*}
 Con\left(
\begin{tikzpicture}
[baseline, my circle/.style={draw, fill, circle, minimum size=3pt, inner sep=0pt}, level distance=0.5cm, 
level 2/.style={sibling distance=0.6cm}, sibling distance=0.9cm,baseline=4.5ex]
\node [my circle] {} [grow=up]
{
	child {node [my circle]  {} 
		child {node [my circle,label=right:\tiny{$5$}] {} child {node [my circle,label=right:\tiny{$4$}] {}}} 
	child {node [my circle,label=right:\tiny{$3$}] {}  }}
	child {node [my circle] {}  child {node [my circle,label=left:\tiny{$2$}] {} } 
	child {node [my circle,label=left:\tiny{$1$}] {} }}
};
\draw [-] (0,-.25) -- (0,0) ;
\end{tikzpicture} 
\right) = 
\begin{tikzpicture}
[baseline, my circle/.style={draw, fill, circle, minimum size=3pt, inner sep=0pt}, level distance=0.5cm, 
level 2/.style={sibling distance=0.6cm}, sibling distance=0.7cm,baseline=3.5ex]
\node [my circle] {} [grow=up]
{child {node [my circle,label=right:\tiny{$5$}] {} child {node [my circle,label=right:\tiny{$4$}] {}}}
	child {node [my circle,label=right:\tiny{$3$}] {}  }
	  child {node [my circle,label=left:\tiny{$2$}] {} } 
	child {node [my circle,label=left:\tiny{$1$}] {} }};
\draw [-] (0,-.25) -- (0,0) ;
\end{tikzpicture}  .
\end{equation*}

\subsubsection{The operad $\ca{LW}$}

For $m\geq 1$, $n\geq 2$  and  $1\leq i \leq m$, let
\begin{equation}\label{eq: partial right mod}
\hat{\circ}_i\co \ca{LW}(m)\ot \mathcal{W}(n) \to \ca{LW}(m+n-1) 
\end{equation}
be linear map defined as follows. 
Given $T$ in $\ca{LW}(m)$ and $R$ in $\ca{W}(n) $, for $1\leq i\leq m$, let $v$ be the vertex labeled by $i$. We set 
\begin{equation}
T\hat{\circ}_i R= 
					\begin{cases}
					\sum_{\sigma \in {\mathbb S}_{V_{unl}(R)}} \epsilon(\sigma) Con_r( T\circ_i Con(R_{\sigma})) &
					\text{\small if the $\prec$--successor of $v$ exists and  is labeled}; \\
					T\circ_i R & \text{\small otherwise}, 
					\end{cases}
\end{equation} 
where $r$ denotes the root vertex of $Con(R_{\sigma})$ seen as a sub-tree of $T\circ_i Con(R_{\sigma})$ and $V_{unl}(R)$ is the set of unlabeled vertices of $R$.  

For instance,
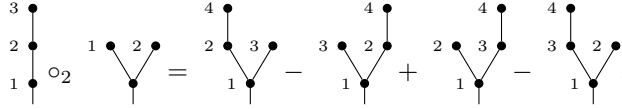
\[\begin{tikzpicture}
[my circle/.style={draw, fill, circle, minimum size=3pt, inner sep=0pt},  
level 2/.style={sibling distance=0.6cm,level distance=0.5cm}, sibling distance=0.6cm, level 1/.style={level distance=0.4cm},baseline=3ex]
\node {} [grow=up]
{child{node [my circle,label=left:\tiny{$1$}] {} child {node (A) [my circle,label=left:\tiny{$2$}]  {}
			child {node [my circle,label=left:\tiny{$3$}] {}}} }}; 
\end{tikzpicture}
\circ_2
\begin{tikzpicture}
[my circle/.style={draw, fill, circle, minimum size=3pt, inner sep=0pt},  
level 2/.style={sibling distance=0.6cm,level distance=0.5cm}, sibling distance=0.6cm, level 1/.style={level distance=0.4cm},baseline=3ex]
\node {} [grow=up]
{child{node [my circle] {} child {node (A) [my circle,label=left:\tiny{$2$}]  {}
		} child {node [my circle,label=left:\tiny{$1$}] {}} }};
\end{tikzpicture} 
=
\begin{tikzpicture}
[my circle/.style={draw, fill, circle, minimum size=3pt, inner sep=0pt},  
level 2/.style={sibling distance=0.6cm,level distance=0.5cm}, sibling distance=0.6cm, level 1/.style={level distance=0.4cm},baseline=3ex]
\node {} [grow=up]
{child{node [my circle, label=left:\tiny{$1$}] {} child {node (A) [my circle,label=left:\tiny{$3$}]  {}
		} child {node [my circle,label=left:\tiny{$2$}] {} child {node [my circle,label=left:\tiny{$4$}] {}}} }};
\end{tikzpicture} 
-
\begin{tikzpicture}
[my circle/.style={draw, fill, circle, minimum size=3pt, inner sep=0pt},  
level 2/.style={sibling distance=0.6cm,level distance=0.5cm}, sibling distance=0.6cm, level 1/.style={level distance=0.4cm},baseline=3ex]
\node {} [grow=up]
{child{node [my circle, label=left:\tiny{$1$}] {} child {node (A) [my circle,label=left:\tiny{$2$}] {} child {node [my circle,label=left:\tiny{$4$}] {}}}
		child {node  [my circle,label=left:\tiny{$3$}]  {}
}  }};
\end{tikzpicture} 
+
\begin{tikzpicture}
[my circle/.style={draw, fill, circle, minimum size=3pt, inner sep=0pt},  
level 2/.style={sibling distance=0.6cm,level distance=0.5cm}, sibling distance=0.6cm, level 1/.style={level distance=0.4cm},baseline=3ex]
\node {} [grow=up]
{child{node [my circle, label=left:\tiny{$1$}] {} child {node (A) [my circle,label=left:\tiny{$3$}]  {}
			child {node [my circle,label=left:\tiny{$4$}] {}}} child {node [my circle,label=left:\tiny{$2$}] {} } }};
\end{tikzpicture} 
-
\begin{tikzpicture}
[my circle/.style={draw, fill, circle, minimum size=3pt, inner sep=0pt},  
level 2/.style={sibling distance=0.6cm,level distance=0.5cm}, sibling distance=0.6cm, level 1/.style={level distance=0.4cm},baseline=3ex]
\node {} [grow=up]
{child{node [my circle, label=left:\tiny{$1$}] {} child {node (A) [my circle,label=left:\tiny{$2$}] {} }
		child {node  [my circle,label=left:\tiny{$3$}]  {}
			child {node [my circle,label=left:\tiny{$4$}] {}}}  }};
\end{tikzpicture} .\]

On the other hand, note that \eqref{eq: partial compo partial planar tree} restricts to 
\begin{equation}\label{eq: part compo LW L}
\circ_i\co \mathcal{LW}(m)\ot \mathcal{L}(n) \to \mathcal{LW}(m+n-1).
\end{equation}

\begin{prop}
 $\ca{LW}$ endowed with the partial compositions defined in \eqref{eq: part compo LW L} and \eqref{eq: partial right mod} is an operad. 
\end{prop}
\begin{proof}
	It is a tedious but direct verification. For the reader convenience we sketch the proof of the nested associativity of the partial compositions, i.e.  we show that
	for every three trees $R, S$ and $T$ in $\ca{W}(r)$, $\ca{W}(s)$, and, respectively, in $\ca{W}(t)$
	\begin{equation*}
	 R \hat{\circ}_i( S \hat{\circ}_j T) = 	(R \hat{\circ}_i S) \hat{\circ}_{i-1+j} T,  
	\end{equation*}   
	for all $1\leq i\leq r$ and $1\leq j\leq s$. 
	
	Let $v$ and $w$ be the vertices of $R$ and $S$ labeled, respectively, by $i$ and $j$. 
	If both $v$ and $w$ have no labeled $\prec$--successor, then the involved partial compositions are the ones of the operad $\Tp$, see \eqref{eq: partial compo partial planar tree}, which proves the equality. 
	Now, let us suppose both $v$ and $w$ have a labeled $\prec$--successor (the proof for the other cases is similar).  
  In this case, recall that one has 
\begin{equation}\label{eq: compo contr}
S\hat{\circ}_jT= \sum_{\sigma \in {\mathbb S}_{V_{unl}(T)}} \epsilon(\sigma) Con_r( S\circ_j Con(T_{\sigma})) .
\end{equation} 
First, set  $P\widetilde{\circ}_j Q:=  Con_r( P\circ_j Con(Q))$ for every couple of trees $P,Q\in \ca{W}$ and observe that the equality $R {\circ}_i( S {\circ}_j T) = (R {\circ}_i S) {\circ}_{i-1+j} T$  in $\Tp$ implies that  
   \begin{equation*}
   R \widetilde{\circ}_i( S \widetilde{\circ}_j T) = 	(R \widetilde{\circ}_i S) \widetilde{\circ}_{i-1+j} T.
   \end{equation*}
Indeed, note that the second contraction $Con_r$ in \eqref{eq: compo contr} does not affect the associativity of the $\circ_k$'s since the sets of the minimal labeled angles of $Con_r(S\circ_j Con(T_{\sigma}))$ and of $S\circ_j Con(T_{\sigma})$ are the same, and  the set of inputs of $v$ and of $w$ are not concerned by these contractions. Similarly, as the contraction $Con(T)$ involves only edges between unlabeled vertices, it has no consequence on the associativity of the $\circ_k$'s. 
We conclude by observing that, since $w$ has a labeled $\prec$--successor, one has 
$V_{unl}( Con_r(S\circ_i Con(T_{\sigma}))= V_{unl}(S)$. 
\end{proof}

Let $\mathcal{I}\subset \mathcal{LW}$ be the operadic ideal generated by 
\[\Biggl \{ \begin{tikzpicture}
[baseline, my circle/.style={draw, fill, circle, minimum size=3pt, inner sep=0pt},  
level 2/.style={sibling distance=0.6cm,level distance=0.5cm}, sibling distance=0.6cm, level 1/.style={level distance=0.5cm}]
\node [my circle]  {} [grow=up]  {child {node (A) [my circle,label=left:\tiny{$2$}]  {}} 
child {node (B) [my circle,label=left:\tiny{$1$}] {}} };
\draw [-] (0,-.25) -- (0,0) ;

\end{tikzpicture} 
-
\begin{tikzpicture}
[baseline, my circle/.style={draw, fill, circle, minimum size=3pt, inner sep=0pt},  
level 2/.style={sibling distance=0.6cm,level distance=0.5cm}, sibling distance=0.6cm, level 1/.style={level distance=0.5cm}]

\node [my circle]  {} [grow=up] {child {node (A) [my circle,label=left:\tiny{$1$}]  {}
		} child {node  [my circle,label=left:\tiny{$2$}] {}} };
\draw [-] (0,-.25) -- (0,0) ;
\end{tikzpicture} ,
\begin{tikzpicture}
[baseline, my circle/.style={draw, fill, circle, minimum size=3pt, inner sep=0pt},  
level 2/.style={sibling distance=0.6cm,level distance=0.5cm}, sibling distance=0.6cm, level 1/.style={level distance=0.5cm}]

\node [my circle] {} [grow=up] { child {node (A)  [my circle] {}  child {node  [my circle,label=left:\tiny{$3$}]  {}}
			child {node  [my circle,label=left:\tiny{$2$}]  {}
		} } child {node  [my circle,label=left:\tiny{$1$}]  {}
} };
\draw [-] (0,-.25) -- (0,0) ;
\end{tikzpicture} 
-
\begin{tikzpicture}
[baseline, my circle/.style={draw, fill, circle, minimum size=3pt, inner sep=0pt},  
level 2/.style={sibling distance=0.6cm,level distance=0.5cm}, sibling distance=0.6cm, level 1/.style={level distance=0.5cm}]

\node [my circle] {} [grow=up] {child {node (A)  [my circle,label=left:\tiny{$3$}]  {}}
		child {node   [my circle] {}  child {node  [my circle,label=left:\tiny{$2$}]  {}}
			child {node  [my circle,label=left:\tiny{$1$}]  {}
} }  };
\draw [-] (0,-.25) -- (0,0) ;
\end{tikzpicture} 
-
\begin{tikzpicture}
[baseline, my circle/.style={draw, fill, circle, minimum size=3pt, inner sep=0pt},  
level 2/.style={sibling distance=0.6cm,level distance=0.5cm}, sibling distance=0.6cm, level 1/.style={level distance=0.5cm}]

\node [my circle] {}  [grow=up] {child {node (A)  [my circle] {}  child {node  [my circle,label=left:\tiny{$3$}]  {}}
			child {node  [my circle,label=left:\tiny{$1$}]  {}
		} } child {node  [my circle,label=left:\tiny{$2$}]  {}
} };
\draw [-] (0,-.25) -- (0,0) ;
\end{tikzpicture} 
 \Biggr \}. \]
 \begin{defn}\label{def: operad TBr}
 We call $\TB$ the symmetric operad $ \mathcal{LW}/ \mathcal{I}$.  
\end{defn}

For later use, we prove the following result. 

\begin{lem} \label{lem: decompo gen of LW}
	For $m\geq 2$, every generator of the $\mathbb{S}_m$--module $\ca{LW}(m)$ has a unique $\mathbb{S}$--equivariant decomposition of the form 
	\begin{equation}
	(\cdots(\cdots ((S_1 \circ_{i_1} S_2 ) \circ_{i_2} S_3) \cdots S_{k-1} )\circ_{i_{k-1}} S_k)^{\sigma},
	\end{equation}
	  where $S_j \in \ca{LW}(2)$ for $1\leq j\leq k$ and  $\sigma \in {\mathbb S}_m$. 
\end{lem}
\begin{proof}
Recall that every labeled tree $T \in \ca{L}(m)$ has a unique $\mathbb{S}$--equivariant decomposition into corollas, that is, there is a $k\geq 1$ and there are corollas $Q_j \in \ca{L}(m_j)$ for $1 \leq j \leq k$ such that
\[
T = (\cdots(\cdots ((Q_1 \circ_{i_1} Q_2 ) \circ_{i_2} Q_3) \cdots Q_{k-1})\circ_{i_{k-1}} Q_k).
\] 
The corollas $Q_j \in \ca{L}(m_j)$ and the $i_j$'s are unique up to the action of symmetric groups $\mathbb{S}_{m_j}$. 
Let us prove the lemma for corollas in $\ca{L}(n)$. 
One has 
\begin{equation}
\begin{tikzpicture}
[baseline, my circle/.style={draw, fill, circle, minimum size=3pt, inner sep=0pt}, level distance=0.5cm, 
level 2/.style={sibling distance=0.8cm}, sibling distance=0.6cm]
\node [my circle,label=left:\tiny{$3$}] {} [grow=up]
{ child {node [my circle,label=left:\tiny{$2$}]  {}}
	child {node [my circle,label=left:\tiny{$1$}]  {}}};
\draw [-] (0,-.25) -- (0,0) ;
\end{tikzpicture} =\begin{tikzpicture}
[baseline, my circle/.style={draw, fill, circle, minimum size=3pt, inner sep=0pt}, level distance=0.5cm, 
level 2/.style={sibling distance=0.8cm}, sibling distance=0.6cm]
\node [my circle,label=left:\tiny{$2$}] {} [grow=up]
{child {node [my circle,label=left:\tiny{$1$}]  {}} };
\draw [-] (0,-.25) -- (0,0) ;
\end{tikzpicture} \circ_2 \begin{tikzpicture}
[baseline, my circle/.style={draw, fill, circle, minimum size=3pt, inner sep=0pt}, level distance=0.5cm, 
level 2/.style={sibling distance=0.8cm}, sibling distance=0.6cm]
\node [my circle,label=left:\tiny{$2$}] {} [grow=up]
{child {node [my circle,label=left:\tiny{$1$}]  {}} };
\draw [-] (0,-.25) -- (0,0) ;
\end{tikzpicture} - \begin{tikzpicture}
[baseline, my circle/.style={draw, fill, circle, minimum size=3pt, inner sep=0pt}, level distance=0.5cm, 
level 2/.style={sibling distance=0.8cm}, sibling distance=0.6cm]
\node [my circle,label=left:\tiny{$2$}] {} [grow=up]
{child {node [my circle,label=left:\tiny{$1$}]  {}} };
\draw [-] (0,-.25) -- (0,0) ;
\end{tikzpicture} \circ_1 \begin{tikzpicture}
[baseline, my circle/.style={draw, fill, circle, minimum size=3pt, inner sep=0pt}, level distance=0.5cm, 
level 2/.style={sibling distance=0.8cm}, sibling distance=0.6cm]
\node [my circle,label=left:\tiny{$2$}] {} [grow=up]
{child {node [my circle,label=left:\tiny{$1$}]  {}} };
\draw [-] (0,-.25) -- (0,0) ;
\end{tikzpicture}
.
\end{equation}
So the result holds for $n=2$. 
For $n\geq 2$, one has  
\begin{equation*}
\begin{tikzpicture}[baseline, my circle/.style={draw, fill, circle, minimum size=3pt, inner sep=0pt}, level distance=0.5cm, 
level 2/.style={sibling distance=0.8cm}, sibling distance=0.6cm]
\node [my circle,label=left:\tiny{$n+1$}] {} [grow=up]
{child {node [my circle,label=above:\tiny{$n$}]  {}} 
	child {node [label=above:\tiny{$\cdots$}]  {}}
	child {node [my circle,label=above:\tiny{$2$}]  {}}
	child {node [my circle,label=above:\tiny{$1$}]  {}}};
\draw [-] (0,-.25) -- (0,0) ;
\end{tikzpicture}
=
 \begin{tikzpicture}[baseline, my circle/.style={draw, fill, circle, minimum size=3pt, inner sep=0pt}, level distance=0.5cm, 
level 2/.style={sibling distance=0.8cm}, sibling distance=0.6cm]
\node [my circle,label=left:\tiny{$2$}] {} [grow=up]
{child {node [my circle,label=left:\tiny{$1$}]  {}} };
\draw [-] (0,-.25) -- (0,0) ;
\end{tikzpicture} 
\circ_2 
\begin{tikzpicture}[baseline, my circle/.style={draw, fill, circle, minimum size=3pt, inner sep=0pt}, level distance=0.5cm, 
level 2/.style={sibling distance=0.8cm}, sibling distance=0.6cm]
\node [my circle,label=left:\tiny{$n$}] {} [grow=up]
{child {node [my circle,label=above:\tiny{$n-1$}]  {}} 
child {node [label=above:\tiny{$\cdots$}]  {}}
child {node [my circle,label=above:\tiny{$2$}]  {}}
child {node [my circle,label=above:\tiny{$1$}]  {}}};
\draw [-] (0,-.25) -- (0,0) ;
\end{tikzpicture} 
-\sum_{1\leq i \leq n-1} \Biggl (  
\begin{tikzpicture}
[baseline, my circle/.style={draw, fill, circle, minimum size=3pt, inner sep=0pt}, level distance=0.5cm, 
level 2/.style={sibling distance=0.8cm}, sibling distance=0.6cm]
\node [my circle,label=left:\tiny{$n$}] {} [grow=up]
{child {node [my circle,label=above:\tiny{$n-1$}]  {}} 
child {node [label=above:\tiny{$\cdots$}]  {}}
child {node [my circle,label=above:\tiny{$2$}]  {}}
child {node [my circle,label=above:\tiny{$1$}]  {}}};
\draw [-] (0,-.25) -- (0,0) ;
\end{tikzpicture} 
\circ_i  
\begin{tikzpicture}[baseline, my circle/.style={draw, fill, circle, minimum size=3pt, inner sep=0pt}, level distance=0.5cm, 
level 2/.style={sibling distance=0.8cm}, sibling distance=0.6cm]
\node [my circle,label=left:\tiny{$2$}] {} [grow=up]
{child {node [my circle,label=left:\tiny{$1$}]  {}} };
\draw [-] (0,-.25) -- (0,0) ;
\end{tikzpicture}
 \Biggr )^{\sigma_i}, 
\end{equation*}
where $\sigma_i\in {\mathbb S}_{n+1}$ is given by $\sigma_i(1)=i$, $\sigma_i(i-k+1)=i-k$, $\sigma_i(i+m)=i+m$, $1 \leq k \leq i-1$, $1 \leq m \leq n+1-i$.
The result follows by induction on $n$. 

Let us prove the lemma for $\w(m)$. Note that the statement is true for $\w(2)$. 
Let $m\geq 3$ and $T$ be a tree of $\w(m)$. By definition the root of $T$ is unlabeled, so there are two unique trees  $T_1$ and $T_2$ of $\ca{LW}$ such that 
\begin{equation} \label{eq: decompo in L}	
T= \Biggl ( 
\begin{tikzpicture} 
[baseline, my circle/.style={draw, fill, circle, minimum size=3pt, inner sep=0pt}, level distance=0.5cm, 
level 2/.style={sibling distance=0.8cm}, sibling distance=0.6cm]
\node [my circle] {} [grow=up]
{child {node [my circle,label=above:\tiny{$2$}]  {}}
child {node [my circle,label=above:\tiny{$1$}]  {}}};
\draw [-] (0,-.25) -- (0,0) ;
\end{tikzpicture}
\circ_2 T_2 \Biggr )  \circ_1 T_1. 
\end{equation} 
Possibly,  either $T_1$ or $T_2$ is the unit tree of $\ca{L}(1)$, in which case the decomposition simplifies.    
If both $T_1$ and $T_2$ or fully labeled (\ie they are trees of $\ca{L}$), then the result follows from the previous case. 
Otherwise, one proceeds by induction, observing that for the tree(s) $T_i$ of $\w(m_i)$ one has $m_i<m$. 
\end{proof}

\subsection{Relating $\TB$ with other operads}

In this subsection we will make explicit how the operad $\TB$ relates to a few other well known algebraic operads.

\begin{thm}\label{thm: isom}
	The operad $\TB$ is isomorphic to the operad $\ca{P}ost\ca{L}ie$. 
\end{thm}
\begin{proof}
	Recall that the operad   $\ca{P}ost\ca{L}ie=  \ca{F}(E)/(R)$ 
	 is a linear quadratic operad generated by two operations $E=E(2)= \Bbbk[\mathbb{S}_2] \langle [-,-] ,\triangleright \rangle$ and relations $R = R_{Lie} \oplus R_{r} \oplus R_{l}$ where  
	$R_{Lie}$ are the Lie relations (antisymmetry and Jacobi) for $[-,-]$ and  $R_{r}$ and $R_{l}$ correspond to the right and left post-Lie relation  \eqref{eq:pl1} and \eqref{eq:pl2} respectively:  
	\begin{equation*}
	\triangleright \circ_2 [-,-] - [-,-] \circ_1 \triangleright - [-,-] \circ_2 (\triangleright\cdot(21)) \text{ and } 
	\end{equation*}	
	\begin{equation*}
	\triangleright \circ_1 [-,-] - {\mathrm a}_{\triangleright} +{\mathrm a}_{\triangleright}\cdot(213),
	\end{equation*}
	where $ {\mathrm a}_{\triangleright}= \triangleright \circ_1 \triangleright - \triangleright \circ_2 \triangleright$.

Let us define horizontal maps in the diagram 
\begin{equation*}
\begin{tikzpicture}[>=stealth,thick,draw=black!50, arrow/.style={->,shorten >=1pt}, point/.style={coordinate}, pointille/.style={draw=red, top color=white, bottom color=red},scale=1,baseline={([yshift=-.5ex]current bounding box.center)}]
\matrix[row sep=9mm,column sep=26mm,ampersand replacement=\&]
{
	\node (00) {$\ca{F}(E)$} ; \&   \node (01){$\ca{LW}$} ; \\
	\node (10) {$\ca{P}ost\ca{L}ie$} ; \&   \node (11){$\TB=\ca{LW}/\ca{I}$} ; \\
}; 
\path 
(00)     edge[left,->>]      node {$p$}  (10)
([yshift=3pt] 00.east)     edge[above,arrow]      node {$h$}  ([yshift=3pt] 01.west)
([yshift=-3pt] 01.west)     edge[below,arrow]      node {$g$}  ([yshift=-3pt] 00.east)
(01)     edge[right,->>]      node {$\pi$}  (11)	
([yshift=3pt] 10.east)     edge[above,arrow]      node {$\overline{h}$}  ([yshift=3pt] 11.west)
([yshift=-3pt] 11.west)     edge[below,arrow]      node {$\overline{g}$}  ([yshift=-3pt]10.east)
; 
\end{tikzpicture}
\end{equation*} 

Let
\begin{equation*}
	h \co   \ca{F}(E) \to \ca{LW} 
\end{equation*}
be the unique morphism of operads such that 
\begin{equation*}
h(	[-,-])=  
\begin{tikzpicture}
[baseline, my circle/.style={draw, fill, circle, minimum size=3pt, inner sep=0pt}, level distance=0.5cm, 
level 2/.style={sibling distance=0.8c}, sibling distance=0.6cm,baseline=2ex]
\node [my circle, yshift=0.2cm] {} [grow=up]
{child {node [my circle,label=left:\tiny{$2$}]  {}} 
	child {node [my circle,label=left:\tiny{$1$}]  {}} };
\draw [-] (0,-.05) -- (0,0.2) ;
\end{tikzpicture} 
\text{ and  } ~~
{h}(\triangleright) = 
\begin{tikzpicture}
[baseline, my circle/.style={draw, fill, circle, minimum size=3pt, inner sep=0pt}, level distance=0.5cm, 
level 2/.style={sibling distance=0.8cm}, sibling distance=0.6cm,baseline=2ex]
\node [my circle, yshift=0.2cm, label=left:\tiny{$2$}] {} [grow=up]
{ child {node [my circle,label=left:\tiny{$1$}]  {}} };
\draw [-] (0,-.05) -- (0,0.2) ;
\end{tikzpicture}.
\end{equation*}
A direct inspection shows that   $\pi h_3(R_{Jac})=\pi h_3(R_{r}) =\pi h_3(R_{l})  =0$, so $\pi h$ induces a morphism $\overline{h} \co  \mathcal{P}ost\mathcal{L}ie \to \TB$. 
Let 
\begin{equation*}
g_n \co \ca{LW}(n) \to \ca{F}(E)(n)
\end{equation*}
be  as follows. 
Let 
$g_1 (\bullet) = 1 \in \mathbb K =\ca{P}ost\ca{L}ie(1)$ and $g_2 = h_2^{-1}$. 
For $n\geq 3$ and a generator $T \in \ca{LW}(n)$ recall the decomposition of Lemma \ref{lem: decompo gen of LW} and let 
\begin{equation*}
g_{n}(T) :=(\cdots ((g_2(S_1) \circ_{i_1} g_2(S_2) ) \circ_{i_2} \cdots \circ_{i_{k-1}} g_2(S_k))\cdot {\sigma}. 
\end{equation*}
We then extend it by linearity to get $g_n$. 
By uniqueness of the decomposition given in Lemma \ref{lem: decompo gen of LW}, $g_n$ is well-defined and is, moreover, a morphism of operads. 
Also, it is a section of $h_n$ \ie  $h_ng_n=id_n$. 
A direct computation shows that $pg_3(\ca{I})=0$. 
Therefore $p g$ induces a morphism $\overline{g}\co \TB \to \ca{P}ost\ca{L}ie$.

Let us prove that 
\begin{equation*}
\overline{h}_n\overline{g}_n=id_n \text{ and } \overline{g}_n\overline{h}_n=id_n  \text{ for } n\geq 2. 
\end{equation*} 
The first identity follows from  $h_ng_n=id_n$. Indeed, 
for $n\geq 3$ and $[T]\in \TB(n)$, one has 
\begin{align*}
\overline{h}_n \overline{g}_n(T+\ca{I})  =\overline{h}_npg_n(T) = \overline{h}_n(g_n(T)+R) = \pi h_n(g_n(T))  
=\pi (T). 
\end{align*}  
The other identity is obtained similarly: for $n\geq 2$, any $[T] \in \ca{P}ost\ca{L}ie(n)$ can be written as $[T]=\sum k_iT_i + R$, where each $T_i\in \ca{F}(E)(n)$  is a composition of $Q^i_j\in E(2)$. 
 Therefore,  $\overline{h}_{n}([T])=\sum k_ih_{n}(T_i)$,  where each $h_{n}(T_i)$ can be written as a composition, in $\ca{LW}(n)$,  of $h_2(Q^i_j)$. By unicity of the decomposition of Lemma \ref{lem: decompo gen of LW} and by definition of $g_{n}$, one obtains the result. 
\end{proof}

\begin{rem}\label{rem: free postlie}
Given an operad $\mathcal{O}$ and a vector space $V$, we denote by $\mathcal{O}(V)$ the free $\mathcal{O}$--algebra generated by $V$. It is explicitly given by $ \mathcal{O}(V)= \bigoplus_{n\geq 0} \mathcal{O}(n)\ot_{\mathbb{S}_n} V^{\ot n}$. 
By Theorem \ref{thm: isom}, we know that $\TB(\K)$ is the free post-Lie algebra on $\K$, which is the vector space generated by trees of $\TB$, putting aside their labelings. In other words, if we let $\K=\K<\bsq>$ for a generator $\bsq$, then $\TB(\K)$ is generated by the set 
	\begin{equation*}\label{eq: set gen of free postlie}
	\mathcal{G}= \bigg\{
	 \begin{tikzpicture}
	[  level distance=0.5cm, level 2/.style={sibling distance=0.6cm}, sibling distance=0.6cm,baseline=2.5ex,level 1/.style={level distance=0.4cm}]
	\node [] {} [grow=up]
	{	
		child {node [sq]  {}
		}
	};
	\end{tikzpicture} 
	,
	\begin{tikzpicture}
	[  level distance=0.5cm, level 2/.style={sibling distance=0.6cm}, sibling distance=0.6cm,baseline=2.5ex,level 1/.style={level distance=0.4cm}]
	\node [] {} [grow=up]
	{	
		child {node [sq]  {}
			child {node [sq]  {} 
			}
		}
	};
	\end{tikzpicture} 
	,
	\begin{tikzpicture}
	[  level distance=0.5cm, level 2/.style={sibling distance=0.6cm}, sibling distance=0.6cm,baseline=2.5ex,level 1/.style={level distance=0.4cm}]
	\node [] {} [grow=up]
	{	
		child {node [sq]  {}
			child {node [sq]  {} }
			child {node [sq]  {} 
			}
		}
	};
	\end{tikzpicture} 
	,
	\begin{tikzpicture}
	[  level distance=0.5cm, level 2/.style={sibling distance=0.6cm}, sibling distance=0.6cm,baseline=2.5ex,level 1/.style={level distance=0.4cm}]
	\node [] {} [grow=up]
	{	
		child {node [sq]  {}
			child {node [sq]  {} 
				child {node [sq]  {} 
			}}
		}
	};
	\end{tikzpicture} 
	,
	\begin{tikzpicture}
	[  level distance=0.5cm, level 2/.style={sibling distance=0.6cm}, sibling distance=0.6cm,baseline=2.5ex,level 1/.style={level distance=0.4cm}]
	\node [] {} [grow=up]
	{	
		child {node [sq]  {}
			child {node [sq]  {} }
			child {node [sq]  {} 
				child {node [sq]  {} }
			}
		}
	};
	\end{tikzpicture} 
	,
	\begin{tikzpicture}
	[  level distance=0.5cm, level 2/.style={sibling distance=0.6cm}, sibling distance=0.6cm,baseline=2.5ex,level 1/.style={level distance=0.4cm}]
	\node [] {} [grow=up]
	{	
		child {node [sq]  {}
			child {node [sq]  {} 
				child {node [sq]  {} } }
			child {node [sq]  {} }
		}
	};
	\end{tikzpicture} 
	,
	\begin{tikzpicture}
	[  level distance=0.5cm, level 2/.style={sibling distance=0.6cm}, sibling distance=0.6cm,baseline=2.5ex,level 1/.style={level distance=0.4cm}]
	\node [] {} [grow=up]
	{	
		child {node [my circle]  {}
			child {node [sq]  {} }
			child {node [sq]  {} 
				child {node [sq]  {} }
			}
		}
	};
	\end{tikzpicture} 
	,
	\begin{tikzpicture}
	[  level distance=0.5cm, level 2/.style={sibling distance=0.6cm}, sibling distance=0.6cm,baseline=2.5ex,level 1/.style={level distance=0.4cm}]
	\node [] {} [grow=up]
	{	
		child {node [sq]  {}
			child {node [sq]  {} }
			child {node [sq]  {} }
			child {node [sq]  {} }
		}
	};
	\end{tikzpicture} 
	, 
	\begin{tikzpicture}
	[  level distance=0.5cm, level 2/.style={sibling distance=0.6cm}, sibling distance=0.6cm,baseline=2.5ex,level 1/.style={level distance=0.4cm}]
	\node [] {} [grow=up]
	{	
		child {node [sq]  {}
			child {node [sq]  {} 
				child {node [sq]  {} 
					child {node [sq]  {} 
			}}}
		}
	};
	\end{tikzpicture} 
	,\dots \bigg\}.
	\end{equation*}
The Lie product of two generators $R$ and $S$ is the class of the tree $C(\bullet; R,S)$; their post-Lie product $R\triangleright S$ is the grafting obtained from the composition in $\TB$,   
	\begin{equation*}  
	\left(
	\begin{tikzpicture}
	[  level distance=0.5cm, level 2/.style={sibling distance=0.6cm}, sibling distance=0.6cm,baseline=2.5ex,level 1/.style={level distance=0.4cm}]
	\node [] {} [grow=up]
	{	
		child {node [my circle,label=left:\tiny{$2$}]  {}
			child {node [my circle,label=left:\tiny{$1$}]  {} 
			}
		}
	};
	\end{tikzpicture} 
	\circ_2 S \right) \circ_1 R,  
	\end{equation*}
by putting any labeling on $S$ and $R$, composing and then forgetting the labeling. 
\end{rem}
In \cite{CL} Chapoton and Livernet constructed the operad $\ca{RT}$ controlling the symmetric brace algebras together with an isomorphism  $\Phi\co \ca{P}re\ca{L}ie \to  \ca{RT}$. 
As one has a projection $p\co \ca{P}ost\ca{L}ie \to \ca{P}re\ca{L}ie$ defined by killing the Lie bracket $[-,-]$, it is natural to ask if it is possible to recover the operad $\ca{RT}$ from $\TB$.\\     
To answer this question, recall, in short, that the operad $\ca{RT}$ is as follows (see  \cite[Section 1.5]{CL} for details). 
The underlying $\mathbb{S}$--module of $\ca{RT}$ is generated by (isomorphism classes of) labeled non-planar rooted trees. 
The operadic composition is defined similarly to the one in \eqref{eq: partial compo partial planar tree} 
except that, since we work with non-planar trees, the map $\phi$ in \eqref{eq: explicit partial compo partial planar tree} is seen as a map from $In(i)$ to $V(R)$, and the condition about the order when grafting the trees is forgotten.  
\\
Said that, consider the  operadic ideal $\mathcal{J}$ of $\TB$ generated by the class of 
\[C(\bullet;1,2)=
\begin{tikzpicture} 
[baseline, my circle/.style={draw, fill, circle, minimum size=3pt, inner sep=0pt}, level distance=0.4cm, 
level 2/.style={sibling distance=0.8cm}, sibling distance=0.6cm]
\node [my circle] {} [grow=up]
{child {node [my circle,label=above:\tiny{$2$}]  {}}
	child {node [my circle,label=above:\tiny{$1$}]  {}}};
\draw [-] (0,-.25) -- (0,0) ;
\end{tikzpicture}.
\] 

\begin{prop} \label{prop: GO2}
	The operad $\TB / \mathcal{J}$ is canonically isomorphic to the operad $\ca{RT}$. Moreover, the following diagram commutes 
	\begin{equation*}
	\begin{tikzpicture}[>=stealth,thick,draw=black!50, arrow/.style={->,shorten >=1pt}, point/.style={coordinate}, pointille/.style={draw=red, top color=white, bottom color=red},scale=1,baseline={([yshift=-.5ex]current bounding box.center)}]
	\matrix[row sep=9mm,column sep=26mm,ampersand replacement=\&]
	{
		\node (00) {$\ca{P}ost\ca{L}ie$} ; \&   \node (01){$\TB$} ; \\
		\node (10) {$\ca{P}re\ca{L}ie$} ; \&   \node (11){$\ca{RT}$;} ; \\
	}; 
	\path
	(00)     edge[left,->>]      node {$p$}  (10)
	(00)     edge[above,arrow]      node {$\overline{h}$}  (01)
	(01)     edge[right,->>]      node {}  (11)	
	(10)     edge[above,arrow]      node {$\Phi$}  (11)
	; 
	\end{tikzpicture}
	\end{equation*} 
	here the vertical maps are the canonical projections, $\overline{h}$ is the isomorphism of Theorem \ref{thm: isom} and $\Phi$ is the isomorphism from \cite[Theorem 1.9]{CL}. 
\end{prop}

\begin{proof} 
	We only prove the first assertion, as the commutativity of the diagram is a straightforward verification.  
	For all $n$ and $1\leq i \leq n-1$, one has  
	\begin{multline}\label{eq: commut T}
	\begin{tikzpicture}
	[baseline, my circle/.style={draw, fill, circle, minimum size=3pt, inner sep=0pt}, level distance=0.5cm, 
	level 2/.style={sibling distance=0.8cm}, sibling distance=0.9cm]
	\node [my circle,label=left:\tiny{$n$}] {} [grow=up]
	{child {node [my circle,label=right:\tiny{$n-1$}]  {}} 
		child {node [label=above:\tiny{$\cdots$}]  {}}
		child {node [my circle,label=left:\tiny{$2$}]  {}}
		child {node [my circle,label=left:\tiny{$1$}]  {}}};
		\draw [-] (0,-.25) -- (0,0) ;
	\end{tikzpicture} 
	\circ_i 
	\begin{tikzpicture} 
	[baseline, my circle/.style={draw, fill, circle, minimum size=3pt, inner sep=0pt}, level distance=0.4cm, 
	level 2/.style={sibling distance=0.8cm}, sibling distance=0.6cm]
	\node [my circle] {} [grow=up]
	{child {node [my circle,label=above:\tiny{$2$}]  {}}
		child {node [my circle,label=above:\tiny{$1$}]  {}}};
	\draw [-] (0,-.25) -- (0,0) ;
	\end{tikzpicture} 
	= 
	\\
	\begin{tikzpicture}
	[baseline, my circle/.style={draw, fill, circle, minimum size=3pt, inner sep=0pt}, level distance=0.5cm, 
	level 2/.style={sibling distance=0.8cm}, sibling distance=0.6cm]
	\node [my circle,label=left:\tiny{$n+1$}] {} [grow=up]
	{child {node [my circle,label=above:\tiny{$n$}]  {}} 
		child {node [label=above:\tiny{$\cdots$}]  {}}
		child {node [my circle,label=above:\tiny{$i+1$}]  {}}
		child {node [my circle,label=above:\tiny{$i$}]  {}}
		child {node [label=above:\tiny{$\cdots$}]  {}}
		child {node [my circle,label=above:\tiny{$2$}]  {}}
		child {node [my circle,label=above:\tiny{$1$}]  {}}};
		\draw [-] (0,-.25) -- (0,0) ;
	\end{tikzpicture} - 
	\begin{tikzpicture}
	[baseline, my circle/.style={draw, fill, circle, minimum size=3pt, inner sep=0pt}, level distance=0.5cm, 
	level 2/.style={sibling distance=0.8cm}, sibling distance=0.6cm]
	\node [my circle,label=left:\tiny{$n+1$}] {} [grow=up]
	{child {node [my circle,label=above:\tiny{$n$}]  {}} 
		child {node [label=above:\tiny{$\cdots$}]  {}}
		child {node [my circle,label=above:\tiny{$i$}]  {}}
		child {node [my circle,label=above:\tiny{$i+1$}]  {}}
		child {node [label=above:\tiny{$\cdots$}]  {}}
		child {node [my circle,label=above:\tiny{$2$}]  {}}
		child {node [my circle,label=above:\tiny{$1$}]  {}}};
		\draw [-] (0,-.25) -- (0,0) ;
	\end{tikzpicture} 
	\in \mathcal{J}(n+1).
	\end{multline}
	Since $\mathcal{J}$ is an ideal, $\mathcal{J}(n+1)$ is closed under the action of the ${\mathbb S}_{n+1}$. 
	In particular  
	\begin{equation*}
	\begin{tikzpicture}
	[baseline, my circle/.style={draw, fill, circle, minimum size=3pt, inner sep=0pt}, level distance=0.55cm, 
	level 2/.style={sibling distance=0.85cm}, sibling distance=0.6cm]
	\node [my circle,label=left:\tiny{$n+1$}] {} [grow=up]
	{child {node [my circle,label=above:\tiny{$\sigma(n)$}]  {}} 
		child {node [label=above:\tiny{$\cdots$}]  {}}
		child {node [my circle,label=above:\tiny{$\sigma(i+1)$}]  {}}
		child {node [my circle,label=above:\tiny{$\sigma(i)$}]  {}}
		child {node [label=above:\tiny{$\cdots$}]  {}}
		child {node [my circle,label=above:\tiny{$\sigma(2)$}]  {}}
		child {node [my circle,label=above:\tiny{$\sigma(1)$}]  {}}};
		\draw [-] (0,-.25) -- (0,0) ;
	\end{tikzpicture} - 
	\begin{tikzpicture}
	[baseline, my circle/.style={draw, fill, circle, minimum size=3pt, inner sep=0pt}, level distance=0.55cm, 
	level 2/.style={sibling distance=0.85cm}, sibling distance=0.65cm]
	\node [my circle,label=left:\tiny{$n+1$}] {} [grow=up]
	{child {node [my circle,label=above:\tiny{$\sigma(n)$}]  {}} 
		child {node [label=above:\tiny{$\cdots$}]  {}}
		child {node [my circle,label=above:\tiny{$\sigma(i)$}]  {}}
		child {node [my circle,label=above:\tiny{$\sigma(i+1)$}]  {}}
		child {node [label=above:\tiny{$\cdots$}]  {}}
		child {node [my circle,label=above:\tiny{$\sigma(2)$}]  {}}
		child {node [my circle,label=above:\tiny{$\sigma(1)$}]  {}}};
		\draw [-] (0,-.25) -- (0,0) ;
	\end{tikzpicture} 
	\end{equation*} 
	belongs to  $\mathcal{J}(n+1)$ for all $\sigma \in {\mathbb S}_{n}$.  
	Using this observations it is not difficult to see that the class in $\TB / \mathcal{J}$  of any corolla is invariant under permutation of its leaves \ie the class of any corolla may be seen as non-planar corolla.   
	
	As every tree of $\TB$ decomposes into corollas, its class in $\TB/ \mathcal{J}$ may be seen as a non-planar tree. 
	Moreover, recall that $\TB=\ca{LW}/ \ca{I}$. Since $\ca{W}$ is generated, as an operad, by $C(\bullet;1,2)$ 
	(see Lemma \ref{lem: decompo gen of LW}), the trees of $\lt / \mathcal{J}$ are (fully) labeled. 
	From this follows the existence of a canonical bijection between $\lt / \mathcal{J}$ and  $\ca{RT}$. 
	It is straightforward to check that this bijection is compatible with the  operadic structures. 
\end{proof}

\section{The $\TB$--algebras}\label{sec:tbalgebra}
We start this section by characterizing the $\TB$--algebras. 
In particular we show that $\TB$ naturally encodes operations that are iterations of the post-Lie product, mimicking the case of the symmetric braces which are iterations of the pre-Lie product. We term these iterated operations \emph{post-symmetric braces} and the corresponding algebras \emph{post-symmetric brace algebras}.
Then we turn our attention to the universal enveloping algebra of a post-Lie algebra and we show how its structure of $D$--bialgebra arises from the corresponding post-symmetric braces. 

\subsection{Definition of $\TB$--algebras}
Recall the notation for the shuffle coproduct of $T(V)$, from Convention \ref{sec: convention operads}: $\Delta_{sh}(X)= \sum_{I\sqcup J} X_I \ot X_J \in T(V) \ot T(V)$ for any monomial $X\in T(V)$.  

\begin{prop}\label{prop:tb}
A  $\TB$--algebra is a pair $(\mathfrak g,T)$ of a Lie algebra $(\g,[-,-])$ together with a linear map $T\in\operatorname{Hom}_{\mathbb K}(T(\g) \ot \g,\mathfrak g)$ satisfying the following four properties. 
\begin{enumerate}[label=(\roman*)] 
	\item\label{eq: for1} Let $\ca{K}$ be the ideal of $T(\g)$ generated by $xy-yx-[x,y]$ for $x,y\in \g$. 
	  For all $z\in\mathfrak g $ one has 
	  \begin{equation}
	  T(\mathcal K;z)=0.\label{eq:c1}
	  \end{equation}
	\item\label{eq:conv}  $T({1};y)=y$, for all $y\in \mathfrak g$, 
	\item\label{eq:for2} For any two monomials $X$ and $Y=y_1\cdots y_m$ in $T(\mathfrak g)$ and $z\in \g$: 
	\begin{equation}
	T(X;T(Y;z))=\sum_{J_0\sqcup \cdots \sqcup J_m=\{1,...,m\}} T(X_{J_0},T(X_{J_1};y_1),\cdots,T(X_{J_m};y_m);z),\label{eq:c2}
	\end{equation}
	where the $J_i$'s form an \emph{ordered partition} and with the convention that $X_{\emptyset}=1\in T(\g)$. 
		\item\label{eq:for3} For all monomial $X\in T(\g)$ and all $x,y\in \g$, one has 
		\begin{equation}
		T(X;[x,y])=\sum_{I\coprod J}\big[T(X_{I};x),T(X_{J};y)\big]\label{eq:c3}
		\end{equation}
	noticing that, by (ii), the two extremal terms  read as $\big[T(X;x),y\big]$ and $[x,T(X;y)]$. 
\end{enumerate}
	\end{prop}

\begin{proof}
	Since $\TB$ is the quotient of $\lw$ by the Lie ideal $\ca{I}$ (see Definition \ref{def: operad TBr}), it is clear that the corolla $C(\bullet; 1,2)$ endows $\g$ with a Lie bracket.  
	
	 Let us shorten the notation $C(n+1;1,...,n)$  to $C_{n+1}$. For $n\geq 0$, denote by $T_n\co \g^{\ot n} \ot \g \to \g$ the operation corresponding to $C_{n+1}$. They induce an operation $T\co  T(\g) \ot \g \to \g$. 
	\\
	
(i) is a direct consequence of \eqref{eq: commut T}.  
	
(ii) means that $T_0\co \g \to \g$ is the identity map; again this is immediate from its definition. 

Let us show (iii). 
  For $n,m\geq 0$, the element $C_{n+1}\circ_{n+1} C_{m+1} \in \ca{L}(n+m+1)$ is, by definition of the operadic composition, given by  
  \begin{equation}\label{eq: flower tree }
  \sum_{J_1\sqcup \cdots \sqcup J_{m+1}=\{1<...<n\} }
  	\begin{tikzpicture} 
  	[my circle/.style={draw, fill,thick, circle, minimum size=3pt, inner sep=0pt}, level distance=0.5cm, 
  	level 2/.style={sibling distance=0.8cm}, sibling distance=0.6cm,
  	emph/.style={edge from parent/.style={black,thick,draw}},
  	gre/.style={edge from parent/.style={gray,thick,draw}}, baseline=1ex]
  	]
  	\node [my circle,label=left:\tiny{$n+m+1$}] {} [grow=up]
  	{child {node [gray,thick,my circle,label=right:\tiny{$n+m$}] (last) {}
  			child[gray,thick] {node [my circle,label=above:\tiny{}] (e) {}}
  			child[gray,thick] {node [my circle,label=above:\tiny{}] (d) {}}
  		} 
  		child {node []  {}}
  		child {node []  {}
  		}
  		child {node [gray,thick,my circle,label=right:\tiny{$n+1$}] (pen)  {}
  			child {[thick,gray,thick] node [my circle,label=above:\tiny{}] (c) {}}
  			child {[gray,thick] node [my circle,label=above:\tiny{}] (b) {}}
  		}
  		child {[gray,thick] node [my circle]  (k) {}}
  		child {[gray,thick] node []  {}}
  		child {[gray,thick] node [my circle]  (a) {}}
  	};
  	\draw [-] (0,-.25) -- (0,0) ;
  	\draw ($(last)!.3!(pen)$)   node []   {$\cdots$}; 
  	\draw ($(a)!.5!(k)$)   node []   {$\cdots$}; 
  	\draw ($(b)!.5!(c)$)   node []   {$\cdots$}; 
  	\draw ($(d)!.5!(e)$)   node []   {$\cdots$}; 
  	\draw[gray,thick,decorate,decoration={brace,amplitude=10pt}] ($(a)+(0,.1)$.north ) -- ( $(k)+(0,.1)$.north);
  	\draw[gray,thick,decorate,decoration={brace,amplitude=10pt}] ($(b)+(0,.1)$.north ) -- ( $(c)+(0,.1)$.north);
  	\draw[gray,thick,decorate,decoration={brace,amplitude=10pt}] ($(d)+(0,.1)$.north ) -- ( $(e)+(0,.1)$.north);
  	\draw[gray,thick] ($(k)!.5!(a)$)  node [above,yshift=11pt]  {$J_{m+1}$}; 
  	\draw[gray,thick] ($(c)!.5!(b)$)  node [above,yshift=11pt]  {$J_{1}$}; 
  	\draw[gray,thick] ($(e)!.5!(d)$)  node [above,yshift=11pt]  {$J_{m}$}; 
  	\end{tikzpicture}
  	.
  \end{equation} 
  The sum runs over all partitions of $\{1<...<n\}$ by possibly empty order sets $J_i$; they correspond to the fibers of the maps of sets $\phi\co \{1<...<n\}\to Ang_{min}(C_{m+1})$ in \eqref{eq: partial compo partial planar tree}.  
  The leaves of the (gray thick) corolla with root $n+i$ are labeled by $J_i$. 
  
   Denote by $k_i=|J_i|$ for $1\leq i\leq m$ and $k=|J_{m+1}|$. 
  Note that every tree of the sum  \eqref{eq: flower tree } can be written as  
  \begin{equation}\label{eq: flower tree 2}
 \left( \cdots (C_{k+m+1} \circ_{k+1} C_{k_1+1}  )\circ_{k+2} \cdots \circ_{k+m} C_{k_m+1} \right) \cdot \sigma_J, 
  \end{equation}
  where $\sigma_J\in {\mathbb S}_{m+n+1}$ is given by 
  	\begin{equation*}
  	 \scalebox{.9}{ 
  \begin{tikzpicture}[>=stealth,thick,draw=black!50, arrow/.style={->,shorten >=1pt}, point/.style={coordinate}, pointille/.style={draw=red, top color=white, bottom color=red},scale=.8,baseline={([yshift=-.5ex]current bounding box.center)}]
 \matrix [{matrix of math nodes}, column sep=5pt, row sep=1pt,
 left delimiter=(,right delimiter=),ampersand replacement=\&] (m)
 {
 	\underbrace{1 \cdots k} \& \underbrace{k+1 \cdots k_1+k} \& k_1+k+1 	\& \cdots \& \underbrace{\kappa \cdots \kappa+ k_m} \& \kappa+ k_m+1 \& n+m+1 \\
    J_{m+1} \& J_1 				\&  n+1 	\& \cdots  			  \&   J_m \& n+m		\&  n+m+1  \\
 }; 
  \end{tikzpicture}.
}
    	\end{equation*}
    	Here $\kappa= k+k_1+2+k_2+2+...+k_{m-1}+2$. 
    Since $T_n\co \mathfrak{g}^{\ot n+1} \to \mathfrak{g}$ is the operation corresponding to $C_{n+1}$, we see that  \eqref{eq: flower tree 2} provides the result.  
    
   The proof of (iv) is similar to that of (iii), by computing $C_{n+1}\circ_{n+1} C(\bullet;1,2) \in \lw(n+2)$. 
   To end the proof, recall that every tree of $\ca{W}$ decomposes into corollas, and every tree of $\ca{L}$ decomposes as in \eqref{eq: decompo in L}. This essentially says that the relations (i)-(iv) generates any other. 
\end{proof}

\begin{rem}\label{rem:tbr} A few comments are in order.
\begin{enumerate}
\item As remarked in the proof of the theorem above, the morphism $T$ defines, by restriction to the $\mathfrak g^{\otimes n}$'s, a family of linear maps $\{T_{n}\}_{n\geq 1}$, where $T_{n}\in\operatorname{Hom}_{\mathbb K}(\mathfrak g^{\otimes n+1},\mathfrak g)$. Obviously the properties of $T$ expressed by the Formulas \eqref{eq:c1}, \eqref{eq:c2} and \eqref{eq:c3} correspond to conditions on the $T_n$'s. 
For example the compatibility between $T$ and the Lie bracket expressed in \eqref{eq:c1}, when read on the level of the $T_n$'s becomes:
\begin{eqnarray*}
T_{n}(x_1, \ldots, [x_i,x_{i+1}],\ldots; x_{n+2}) &=&  
			T_{n+1}(x_1, \ldots, x_i,x_{i+1},\ldots; x_{n+2})\\
			&-& T_{n+1}(x_1, \ldots,x_{i+1}, x_i,\ldots; 
			x_{n+2}).
\end{eqnarray*}
In Proposition \ref{prop:equivt} will be analyzed in more details how the conditions on the linear map $T$ can be translated to the level of the family of the $\{T_n\}_{n\geq 1}$.

\item 
Applied on $X=x\in\mathfrak g$ and $Y=y_1y_2\in\mathfrak g^{\otimes 2}$, formula \eqref{eq:c2} gives 
\begin{equation*}
T(x;T(y_1,y_2;z))=T(x,y_1,y_2;z)+T(T(x;y_1),y_2;z) + T(y_1,T(x;y_2);z), 
\end{equation*}
for all  $z\in\mathfrak g$.

For $X=xy\in \g^{\otimes 2}$ and $u,v\in \g$, equation \eqref{eq:c3} gives 
\begin{equation*}
T(x,y;[u,v])=[T(x,y;u),v]+[T(x;u),T(y;v)]+[T(y;u),T(x;v)]+[u,T(x,y;v)].
\end{equation*}
\end{enumerate}
\end{rem}

Using the family of linear maps $\{T_{n}\}_{n\geq 1}$, the isomorphism $\TB\cong \ca{P}ost\ca{L}ie$ of Theorem \ref{thm: isom} says that the $\TB$--algebras can be characterized as follows.

\begin{prop}\label{prop:equivt}
	A structure of a $\TB$--algebra on a vector space $V$ is the data of:
	\begin{itemize}
		\item a post-Lie algebra $(V,\triangleright,[-,-])$; and, 
		\item  a family $\{T_{n}\}_{n\geq 1}$ of linear maps $T_{n}\in\operatorname{Hom}_{\mathbb K}(V^{\otimes n+1},V)$ such that for each $n\geq 2$
		\begin{eqnarray}\label{eq:req}
		T_{n}(x_1,\dots,x_n;y)&=&T_{1}(x_1;T_{n-1}(x_2,\dots,x_n;y))\\
		&-&\sum_{k=2}^nT_{n-1}(x_2,\dots,T_{1}(x_1;x_k),\dots,x_n;y)\nonumber,
		\end{eqnarray}
		and $T_{1}=\triangleright$. 
	\end{itemize} 
\end{prop}
\begin{proof}
Let $(\mathfrak g,T)$ be a $\TB$--algebra and let $\{T_{n}\}_{n\geq 1}$ be the corresponding family of linear operators, see Remark \ref{rem:tbr}. Let $X=x\in \mathfrak g$ and $Y=x_1\cdots x_{n-1}\in \mathfrak g^{\otimes n-1}$. 
For all $y\in \g$, one has 
\[
T(X;T(Y;y))=T_{1}(x;T_{n-1}(x_1,\dots,x_{n-1};y)). 
\]  
Since 
\begin{equation}
\Delta_{sh}^{n-1} x=\sum_{i=1}^nx_{(i)},\label{eq:deltax}
\end{equation}
 where $x_{(i)}=1\otimes\cdots\otimes x\otimes\cdots\otimes 1\in \mathfrak g^{\otimes n}$, with $x$ in the position $i$, the term  $T(X;T(Y;y))$ can be written as
\[
T_{n}(x,x_1,\dots,x_{n-1};y)
+\sum_{k=1}^{n-1}T_{n-1}(x_1,\dots,T_{1}(x;x_{k}),\dots,x_{n-1};y), 
\]
see item $(2)$ in Remark \ref{rem:tbr}, which, up to renumbering, is Formula \eqref{eq:req}.
Formulas \eqref{eq:c1} and \eqref{eq:c3} imply that, for all $x,y$ and $z$ in $\mathfrak g$ 
\begin{equation}
T_{1}([x,y];z)=T_{2}(x,y;z)-T_{2}(y,x;z),\label{eq:2}
\end{equation}
and, respectively,
\begin{equation}
T_{1}(x;[y,z])=[T_{1}(x;y),z]+[y,T_{1}(x;z)]\label{eq:1}
\end{equation}
and since $\triangleright=T_{1}$, one concludes that $(\mathfrak g,\triangleright)$ is a post-Lie algebra. On the other hand, suppose that on a post-Lie algebra $(\mathfrak g,\triangleright)$ is defined a family of linear operators $\{T_{n}\}_{n\geq 1}$, where $T_{n}\in\operatorname{Hom}_{\mathbb K}(\mathfrak g^{\otimes n+1},\mathfrak g)$ with $T_{1}=\triangleright$, such that Formula \eqref{eq:req} holds true. Let $T\in\operatorname{Hom}_{\mathbb K}(T(\mathfrak g),\mathfrak g)$ be the unique linear map such that $T(1;y)=y$ for all $y\in\mathfrak g$ and such that, for each $n\geq 1$, its restriction to $\mathfrak g^{\otimes n}$ is equal to $T_{n}$. Equations \eqref{eq:c1}, \eqref{eq:c2} and \eqref{eq:c3} follow from the Formulas \eqref{eq:2} \eqref{eq:req} and, respectively, from \eqref{eq:1} using an inductive argument on the length of the monomial $X$ and $Y$. 
For example, since every $x\in\mathfrak g$ is primitive, for $X=x\in\mathfrak g$, \eqref{eq:c3} is Formula \eqref{eq:1}. Suppose now that \eqref{eq:c3} held for every $X$ in $\mathfrak g^{\otimes k}$, for all $k\leq n-1$ and let $X'=xX$ with $X\in\mathfrak g^{\otimes n-1}$. Then $T(X';[y,z])=T(xX;[y,z])=T_{n}(x,x_1,\dots,x_{n-1};[y,z])$, which, using Formula \eqref{eq:req}, becomes
\[
T_{1}(x;T_{n-1}(x_1,\dots,x_{n-1};[y,z]))-\sum_{k=1}^{n-1}T_{n-1}(x_1,\dots,T_{1}(x;x_k),\dots,x_{n-1};[y,z]).
\]
To conclude the proof it suffices to apply the inductive hypothesis to the terms in the formula above and compare the result obtained with what one gets from the following computation
\[
[T((xX)_{(1)};y),z]+[y,T((xX)_{(2)};z)]
\]
recalling that for all $x\in\mathfrak g$: 
\[
\Delta_{sh}(xX)=(xX)_{(1)}\otimes(xX)_{(2)}=xX_{(1)}\otimes X_{(2)}+X_{(1)}\otimes xX_{(2)}.
\]
One gets Formula \eqref{eq:c1} from Formulas \eqref{eq:2} and \eqref{eq:req} and Formula \eqref{eq:c2} from \eqref{eq:req} using a similar strategy. 
\end{proof}

\subsection{$\TB$-algebras vs $D$--algebras}\label{ssec:tbvsda}

We now discuss briefly some relations between the $\TB$-algebras and the universal enveloping algebras of the corresponding post-Lie algebras. These algebras, which were introduced in \cite{MW} and there termed \emph{$D$--algebras}, were further studied in \cite{MKL}. In what follows we adopt the definition of $D$--algebra proposed in the recent preprint \cite{CEFMMK}.

\begin{defn}[\cite{CEFMMK}]\label{defn:Dalg}
A \emph{$D$--algebra} is a unital, associative algebra $(D,\cdot,1)$ equipped with a non-associative product 
$\triangleright:V\otimes V\rightarrow V$, an exhaustive increasing filtration
\[
\mathbb K\cdot 1=D^0\subset D^1\subset D^2\subset\cdots\subset D^n\subset\cdots
\]
and an augmentation $\epsilon:D\rightarrow\mathbb K$ such that $D^i\cdot D^j\subset D^{i+j}$ and 
\begin{enumerate}
\item[(i)] $1\triangleright X=X$ and $X\triangleright 1=0$, for all $X\in D$.
\item[(ii)] $D_1=\operatorname{ker}(\epsilon)\cap D^1$ is closed with respect to the antisymmetrization of the associative product and with respect to the bilinear product $\triangleright$ and it generates $(D,\cdot)$.
\item[(iii)] $x\triangleright (X\cdot Y)=(x\triangleright X)\cdot Y+X\cdot(x\triangleright Y)$ for all $x\in D_1$ and $X,Y\in D$.
\item[(iv)] $(x\cdot X)\triangleright Y=x\triangleright (X\triangleright Y)-(x\triangleright X)\triangleright Y$, for all $x\in D_1$ and $X,Y\in D$.
\end{enumerate}
\end{defn}

\begin{prop}[\cite{MKL}]
$D_1$ is a post-Lie algebra. 
\end{prop}
\begin{proof}
By antisymmetrizing the associative product $\cdot$ one gets a Lie bracket. Axiom $(ii)$ says that $D_1$ is closed with respect to both this bracket and to the product $\triangleright$. Now axiom $(iii)$ implies \eqref{eq:pl1} and axiom $(iv)$ implies \eqref{eq:pl2}, proving the statement.
\end{proof}

To investigate further the relations between $D$--algebras and $\TB$--algebras, it is convenient to enhance the structure defining a $D$--algebra as follows. 

\begin{defn}\label{def:dbialgebra}
A $D$-bialgebra is a bialgebra $(D,\cdot,1,\Delta,\epsilon)$ endowed with a non-associative product $\triangleright:D\otimes D\rightarrow D$ and an exhaustive, increasing filtration 
\[
\mathbb K\cdot 1=D^0\subset D^1\subset D^2\subset\cdots\subset D^n\subset\cdots
\]
such that $D^i\cdot D^j\subset D^{i+j}$ and
\begin{enumD}
	\item\label{D bial item1} $1\triangleright X=X$ and $X\triangleright 1=0$, for all $X\in D$,
	\item\label{D bial item2} $D_1=\operatorname{ker}(\epsilon)\cap D^1=\operatorname{Prim}(D)$ which generates $(D,\cdot)$,
	\item\label{D bial item3} $\Delta(X\triangleright Y)=(X_{(1)}\triangleright Y_{(1)})\otimes(X_{(2)}\triangleright Y_{(2)})$,
	\item\label{D bial item4} $X\triangleright(Y\cdot Z)=(X_{(1)}\triangleright Y)\cdot(X_{(2)}\triangleright Z)$,
	\item\label{D bial item5} $(x\cdot X)\triangleright y=x\triangleright (X\triangleright y)-(x\triangleright X)\triangleright y$,
	\item\label{D bial item6} $D_1$ is closed under the antisymmetrization of the associative product.
\end{enumD}
In item \ref{D bial item2} $\operatorname{Prim}(D)$ is the vector space of the primitive elements of the coalgebra $(D,\Delta,\epsilon)$ i.e. the set of all $x\in D$ such that $\Delta x=x\otimes 1+1\otimes x$.
\end{defn}

First we prove that any $D$--bialgebra has an underlying structure of a $D$--algebra. 
\begin{prop}
If $(D,\cdot,1,\Delta,\epsilon,\triangleright)$ is a $D$-bialgebra, then $(D,\cdot, 1,\epsilon,\triangleright)$ is a $D$--algebra.
\end{prop}
\begin{proof}
Axioms \ref{D bial item2} and \ref{D bial item3}  imply at once that $D_1$ is closed under the product $\triangleright$. To conclude the proof it suffices to prove that every $D$--bialgebra fulfills the axioms $(iii)$ and $(iv)$ of the $D$--algebras.  
Property $(iii)$ follows at once from \ref{D bial item1}, \ref{D bial item2} and \ref{D bial item4}. 
To prove that in every $D$--bialgebra $(iv)$ holds, let $x\in D_1$ and $X,Y\in D$. If the length of $Y$ is $1$, i.e. if $Y\in D_1$, then $(iv)$ is \ref{D bial item5}. Let $Y'=Y\cdot y$, where length of $Y$ is $n-1$ and suppose that $(iv)$ holds for each $Y$ of length at least $n-1$.
Then
\[
(x\cdot X)\triangleright Y'=(x\cdot X)\triangleright (Y\cdot y)=\big((x\cdot X)_{(1)}\triangleright Y\big)\cdot\big((x\cdot X)_{(2)}\triangleright y\big),
\]
see \ref{D bial item5}. 
Since the coproduct is an algebra morphism and $x$ is primitive, the last term of the previous equality becomes
\[
\big((x\cdot X_{(1)})\triangleright Y\big)\cdot\big(X_{(2)}\triangleright y\big)+\big(X_{(1)}\triangleright Y\big)\cdot\big((x\cdot X_{(2)})\triangleright y\big),
\]
which, by the inductive hypothesis, can be written as
\[
\big(x\triangleright (X_{(1)}\triangleright Y)-(x\triangleright X_{(1)})\triangleright Y\big)\cdot(X_{(2)}\triangleright y)+(X_{(1)}\triangleright Y)\cdot\big(x\triangleright (X_{(2)}\triangleright y)-(x\triangleright X_{(2)})\triangleright y\big).
\] 
Since $x\in D_1$, using \ref{D bial item4} the previous expression can be written as
\[
x\triangleright \big(X\triangleright (Y\cdot y)\big)-[\big((x\triangleright X_{(1)})\triangleright Y\big)\cdot(X_{(2)}\triangleright y)+(X_{(1)}\triangleright Y)\cdot\big((x\triangleright X_{(2)})\triangleright y\big)].
\]
Using \ref{D bial item3}, together with the property of the coproduct of being an algebra morphism, the terms in the square bracket can be written as 
$(x\triangleright X)\triangleright (Y\cdot y)$, giving the proof of the proposition.
\end{proof}

Let  
\begin{equation*}\label{eq:classical univ funct}
\mathcal{U}\co \cat{Lie} \to  \cat{Bialgebra} 
\end{equation*}
be the classical universal enveloping algebra functor. With the following result we show that it enriches to a functor 
\begin{equation}\label{eq:univ funct PSB}
\mathcal{U}\co \cat{PostLie} \to  \cat{D-bialgebra}
\end{equation} 
with adjoint the primitive elements functor defined in \eqref{eq: functor primitive}. 
\begin{thm}\label{thm:tbrD}
	Let $\g$ be a Lie algebra. There is a one-to-one correspondence between the structures of $\TB$--algebra on $\mathfrak g$ and the structures of $D$--bialgebra on $\mathcal U(\mathfrak g)$.
\end{thm}

\begin{proof}
Suppose first that $\mathfrak g$ carries a structure of a $\TB$-algebra. 
 Formula \eqref{eq:c1} implies that $T$ descends to a linear map, still called $T$, from $\mathcal U(\mathfrak g)\otimes\mathfrak g$ to $\mathfrak g$, such that $T(1;x)=x$ for all $x\in\mathfrak g$. Define now $\triangleright:\mathcal U(\mathfrak g)\otimes\mathcal U(\mathfrak g)\rightarrow\mathcal U(\mathfrak g)$ by
\begin{equation}
X\triangleright x=T(X;x),\,\forall X\in\mathcal U(\mathfrak g),\,x\in\mathfrak g,\label{eq:protri1}
\end{equation}
\begin{equation}
X\triangleright 1=0,\,\forall X\in\mathcal U(\mathfrak g),\label{eq:protri2}
\end{equation}
and
\begin{equation}
X\triangleright Y=(X_{(1)}\triangleright y_1)\cdots (X_{(n)}\triangleright y_n)\label{eq:prodtri3}
\end{equation}
for all $Y=y_1\cdots y_n$ and any monomial $X\in\mathcal U(\mathfrak g)$.
Endowing $\mathcal U(\mathfrak g)$ with its standard filtration and its standard bialgebra structure, Properties \ref{D bial item2} and \ref{D bial item6} are automatically fulfilled. Furthermore note that if $X=x\in\mathfrak g$ and $Y=y_1\cdots y_n\in\mathcal U(\mathfrak g)$, \eqref{eq:prodtri3} implies that 
\[
x\triangleright Y=\sum_{i=1}^ny_1\cdots (x\triangleright y_i)\cdots y_n,
\]
i.e. that $\triangleright$ extends to $\mathcal U(\mathfrak g)$ as a derivation of the associative product.
This observation, together with Formula \eqref{eq:c2}, implies that $\triangleright$ satisfies \ref{D bial item5}.
 
Property \ref{D bial item4} follows from Formula \eqref{eq:prodtri3} and from the coassociativity of the coproduct. More precisely for every $n\geq 0$, let 
\[
\nabla_n:\mathcal U(\g)^{\otimes n}\otimes\frac{\mathcal U_{n}(\g)}{\mathcal U_{n-1}(\g)}\rightarrow \frac{\mathcal U_{n}(\g)}{\mathcal U_{n-1}(\g)}
\]
be the linear map defined by
\begin{equation}\label{eq:aux1}
\nabla_n(A_1\otimes\cdots\otimes A_n\otimes x_1\cdots x_n)=(A_1\triangleright x_1)\cdots (A_n\triangleright x_n)
\end{equation}

for $A_1,\dots,A_n\in\mathcal U(\g)$ and $x_1\cdots x_n\in\frac{\mathcal U_{n}(\g)}{\mathcal U_{n-1}(\g)}$.
Note that $\nabla_0$ is the multiplication map of the ground field $\mathbb K$, that $\nabla_1(A\otimes x)=A\triangleright x$, for all $A\in\mathcal U(\g)$ and $x\in\g$ and that, for all $X,Y\in\mathcal U(\g)$, with $Y=y_1\cdots y_n$ 
\begin{equation}
\nabla_n\big(\Delta^{n-1}X\otimes Y\big)=\nabla_n(X_{(1)}\otimes\cdots\otimes X_{(n)}\otimes Y)=X\triangleright Y,\label{eq:rewrit}
\end{equation}
see \eqref{eq:prodtri3}. Furthermore, let
\[
\tau_{n-k,k}:\mathcal U(\g)^{\otimes n}\otimes\frac{\mathcal U_{n}(\g)}{\mathcal U_{n-1}(\g)}\rightarrow\Big(\mathcal U(\g)^{\otimes n-k}\otimes\frac{\mathcal U_{n-k}(\g)}{\mathcal U_{n-k-1}(\g)}\Big)\otimes\Big(\mathcal U(\g)^{\otimes k}\otimes\frac{\mathcal U_{k}(\g)}{\mathcal U_{k-1}(\g)}\Big)
\]
be defined by
\begin{multline}\label{eq:aux2}
\tau_{n-k,k}(A_1\otimes\cdots\otimes A_n\otimes x_1\cdots x_n)=\\
(A_1\otimes\cdots\otimes A_{n-k}\otimes x_1\cdots x_{n-k})\otimes(A_{n-k+1}\otimes\cdots A_n\otimes x_{n-k+1}\cdots x_n)
\end{multline}
and let 
\[
m_{n-k,k}:\frac{\mathcal U_{n-k}(\g)}{\mathcal U_{n-k-1}(\g)}\otimes \frac{\mathcal U_{k}(\g)}{\mathcal U_{k-1}(\g)}\rightarrow \frac{\mathcal U_{n}(\g)}{\mathcal U_{n-1}(\g)}
\]
be the natural multiplication map. One has  
\begin{equation}\label{eq:aux3}
m_{n-k,k}\circ (\nabla_{n-k}\otimes\nabla_k)\circ\tau_{n-k,k}=\nabla_{n}.
\end{equation}
Suppose now that $X,Y,Z\in\mathcal U(\g)$ with $Y=y_1\cdots y_n$ and $Z=z_1\cdots z_k$. 
One has 
\begin{eqnarray*}
&&(X_{(1)}\triangleright Y)\cdot (X_{(2)}\triangleright Z)\stackrel{\eqref{eq:rewrit}}{=}\nabla_n\big(\Delta^{n-1}X_{(1)}\otimes Y\big)\cdot \nabla_k\big(\Delta^{k-1}X_{(2)}\otimes Z\big)\\
&=&m_{n,k}\circ (\nabla_n\otimes\nabla_k)\circ\tau_{n,k}\big(\Delta^{n-1}X_{(1)}\otimes\Delta^{k-1}X_{(2)}\otimes (Y\cdot Z)\big)\\
&\stackrel{\eqref{eq:aux3}}{=}&\nabla_{n+k}\big[\big((\Delta^{n-1}\otimes\Delta^{k-1})\circ\Delta) X\big)\otimes (Y\cdot Z)\big]\\
&=&\nabla_{n+k}(\Delta^{n+k-1}X\otimes (Y\cdot Z),
\end{eqnarray*}
where the last equality follows from the coassociativity of the coproduct. On the other hand, from  
\[
\nabla_{n+k}(\Delta^{n+k-1}X\otimes (Y\cdot Z))\stackrel{\eqref{eq:rewrit}}{=}X\triangleright (Y\cdot Z),
\]
which is what we wanted to show.  
It remains to prove item \ref{D bial item3}. 
If $Y=y\in \g$, then $X\triangleright y$ is primitive, which corresponds to \ref{D bial item3}. 
Suppose by induction that \ref{D bial item3} is true for any monomial $Y=y_1\cdots y_k$ of length $k\leq n$. 

Let $Y'=Y\cdot y$ for $Y$ of length  $ k\leq n$ and $y\in \g$. 
By \eqref{eq:prodtri3}, one has 
\begin{equation*}
\Delta(X\triangleright (Y\cdot y))=\Delta(X_{(1)}\triangleright Y)\cdot\Delta(X_{(2)}\triangleright y)=\Delta(X_{(1)}\triangleright Y)\cdot \big((X_{(2)}\triangleright y)\ot 1 + 1\ot ( X_{(2)}\triangleright y) \big).
\end{equation*}
By induction hypothesis, this gives  
\begin{multline*}
\Delta\big(X\triangleright (Y\cdot y)\big)=
\left( (X_{(1)}\triangleright Y_{(1)})\ot (X_{(2)}\triangleright Y_{(2)})\right)\cdot \big((X_{(3)}\triangleright y)\ot 1 + 1\ot (X_{(3)}\triangleright y )\big)
\\
= \big((X_{(1)}\triangleright Y_{(1)})\cdot(X_{(3)}\triangleright y)\big)\ot (X_{(2)}\triangleright Y_{(2)})
+ (X_{(1)}\triangleright Y_{(1)})\ot \big((X_{(2)}\triangleright Y_{(2)})\cdot(X_{(3)}\triangleright y)\big). 
\end{multline*}
Applying again \eqref{eq:prodtri3}, and noticing that the coproduct is cocommutative, one obtains
\begin{eqnarray*}
\Delta\big(X\triangleright (Y\cdot y)\big)&=&
\big(X_{(1)}\triangleright (Y_{(1)}\cdot y)\big)\ot (X_{(2)}\triangleright Y_{(2)})
+ (X_{(1)}\triangleright Y_{(1)})\ot \big(X_{(2)}\triangleright (Y_{(2)}\cdot y)\big)
\\
&=& (X_{(1)}\triangleright Y'_{(1)})\ot (X_{(2)}\triangleright Y'_{(2)}).
\end{eqnarray*}

Conversely, let $\mathcal U(\mathfrak g)$ be equipped with its standard filtration and  its standard bialgebra structure and supposed it be endowed with a structure of $D$-bialgebra whose $D$-product is
denoted by $\triangleright$. 
\ref{D bial item1}  implies that  $D_1=\mathfrak g$ which, by \ref{D bial item3}, has a structure of post-Lie algebra whose post-Lie product is given by the restriction of $\triangleright$ to $\g\ot \g$. It follows from Theorem \ref{thm: isom} that $\g$ is a $\TB$--algebra, with higher operations given by Proposition \ref{prop:equivt}.  

Now let us see that the these two assignments are inverse to each other. 
In fact, this follows from Theorem \ref{thm: isom} since it says that two $\TB$--algebra structures on the Lie algebra $\g$ with the same post-Lie product are equal. 
\end{proof}

\begin{prop}\label{prop: adjoint}
	The functors $\Prim$ and $\ca{U}$ are adjoints.  
\end{prop}
\begin{proof}
	Let $\g$ be a post-Lie algebra and $D$ be a $D$--bialgebra. 
	Let $f\co \g \to \Prim(D)$ be a morphism of post-Lie algebras. There exists a unique morphism of algebras 
	$\widetilde{f}\co \ca{U}(\g)\to D$ that extends $f$, that is, such that $\Prim(\widetilde{f})\circ \eta=f$ where $\eta\co \g \to \Prim(\ca{U}(\g))$ is the canonical identification. 
	It remains to prove that $\widetilde{f}$ is a morphism of coalgebras and that   $\widetilde{f}(X\triangleright Y)= \widetilde{f}(X)\triangleright \widetilde{f}(Y)$ for all $X,Y\in \ca{U}(\g)$. The first assertion is straightforward, while the second one can be shown by an induction on $n$ in the filtration $\{\ca{U}_n(\g)\}_n$.    
\end{proof}

For a $D$--bialgebra $(D,\cdot,1,\Delta,\epsilon,\triangleright)$, let $\ast\co D\ot D\to D$ be the linear map given by  
\begin{equation}\label{eq:ast prod}
A \ast B :=  A_{(1)}\cdot (A_{(2)} \triangleright B) \text{ for all } A,B \in D.
\end{equation}  
The following is a straightforward, though a bit long, verification. 
\begin{prop}
 $(D,\ast,1,\Delta,\epsilon)$ is a bialgebra. 
\end{prop}
We let 
\begin{equation}
\ast\co \cat{D-bialgebra}\to \cat{Bialgebra}\label{eq:functor ast}
\end{equation}
denote the induced functor. 
	\begin{rem}
		If $x$ belongs to $\Prim(D)$, then one has  
		\begin{equation}
		x\ast X=x\cdot X+x\triangleright X\label{eq:partialfor}
		\end{equation}
		 for all homogeneous element $X\in D$. 
	\end{rem}
\begin{rem}
 The product $\ast$ defined on the universal enveloping algebra of a post-Lie algebra is known as the \emph{Grossmann-Larson product}, see for example \cite{MW, MuntheKaas} and references therein.
\end{rem}

\section{Post-Lie Magnus expansion}\label{sec:magnus}

In this section we further investigate from the view-point of the $\TB$--algebras the so called \emph{post-Lie Magnus expansion}, an interesting series which can be defined in a (suitable completion) of any post-Lie algebra. 
 We start by recalling a few basic properties of post-Lie algebras following \cite{ManVid}, then we introduce the post-Lie Magnus expansion and we discuss its role in the integration of post-Lie algebras.

\subsection{Post-Lie Magnus expansion and BCH formula}

Let $(\g,[-,-],\triangleright)$ be a post-Lie algebra and let $[[-,-]]:\g\otimes\g\rightarrow\g$ the Lie bracket defined in \eqref{eq:newLie}. 
Let $x\in\g$ and define $\nabla:\g\rightarrow\operatorname{End}_{\mathbb K}(\g)$ by
\begin{equation}
\nabla_x(y)=x\triangleright y,\forall y\in\g. 
\end{equation}

\begin{lem}
$\nabla_x$ is a derivation of $(\g,[-,-])$ and $(\nabla,\g)$ is a representation of the Lie algebra $(\g,[[-,-]])$.
\end{lem}
\begin{proof}
The first statement is Formula \eqref{eq:pl2}, while the second follows from a simple computation which we omit.
\end{proof}
From now on $\overline\g$ and $\g$ will denote the Lie algebra $(\g,[[-,-]])$ and, with a slight abuse of notation, the Lie algebra $(\g,[-,-])$. Let $\mathcal U(\overline{\g})$ and $\mathcal U(\g)$ be the universal enveloping algebra of $\overline{\g}$ and of $\g$. Following \cite{ManVid}, for all $x\in\overline{\g}$ let $\sigma_x:\mathcal U(\g)\rightarrow\mathcal U(\g)$ be defined by $\sigma_x(A)=x\cdot A$, where $\cdot$ denotes the associative product in $\mathcal U(\g)$ and let $M:\overline{\g}\rightarrow\operatorname{End}_{\mathbb K}(\mathcal U(\g))$ be defined by
\begin{equation}
M(x)=M_x:=\nabla_x+\sigma_x,\label{eq:mapM}
\end{equation}
for all $x\in\overline{\g}$.
One has
\begin{lem}
$(\mathcal U(\g),M)$ is a representation of $\overline{\g}$, i.e. 
\begin{equation*}
M_{[[x,y]]}=[M_x,M_y],\,\forall x,y\in\overline{\g}.
\end{equation*} 
\end{lem}
\begin{proof}
It suffices to compute 
\begin{eqnarray*}
[M_x,M_y](z)&=&M_x(\nabla_y+\sigma_y)(z)-M_y(\nabla_x+\sigma_x)(z)\\
&=&M_x(y\triangleright z+y\cdot z)-M_y(x\triangleright z+x\cdot z)\\
&=&x\triangleright (y\triangleright z)+x\triangleright (y\cdot z)+x\cdot (y\triangleright z)+x\cdot (y\cdot z)\\
&-&y\triangleright (x\triangleright z)-y\cdot(x\triangleright z)-y\triangleright (x\cdot z)-y\cdot(x\cdot z)\\
&=&x\triangleright (y\triangleright z)-y\triangleright (x\triangleright z)+(x\triangleright    y-y\triangleright x)\cdot z+[x,y]\cdot z,
\end{eqnarray*}
and
\begin{eqnarray*}
M_{[[x,y]]}(z)=(x\triangleright y-y\triangleright x)\triangleright z+[x,y]\triangleright z+(x\triangleright y-y\triangleright x)\cdot z+[x,y]\cdot z,
\end{eqnarray*}
which, thanks to \eqref{eq:pl2}, reduces to the result of the previous computation.
\end{proof}
Using the universal property of the enveloping algebra and the previous lemma, one can extend the application $M$ above defined to an application $M:\mathcal U(\overline{\g})\rightarrow\operatorname{End}_{\mathbb K}\big(\mathcal U(\g)\big)$ which defines on $\mathcal U(\g)$ a structure of (left) $\mathcal U(\overline{\g})$-module. 
More precisely one has the  following.   
\begin{prop}[\cite{ManVid}]\label{pro:impo}
The application $\phi:\mathcal U(\overline{\g})\rightarrow\mathcal U(\g)$, defined by
\begin{equation}
\phi(A)=M_A(1)\label{eq:GOiso}
\end{equation}
for all $A$ monomial in $\mathcal U(\overline{\g})$ and then extended by linearity to $\mathcal U(\overline{\g})$, is both an {\em isomorphism} of (left) $\mathcal U(\overline{\g})$-modules and of coalgebras.
\end{prop}
\begin{proof}
First note that the universal property of the enveloping algebra implies that the application $M:\mathcal U(\overline{\g})\rightarrow\operatorname{End}_{\mathbb K}\big(\mathcal U(\g)\big)$ obtained from the application defined in \eqref{eq:mapM} is a morphism of associative algebras, i.e.
\[
M_{x_1\odot\cdots\odot x_n}=M_{x_1}\circ\cdots\circ M_{x_n},
\]
for all monomials $x_1\odot\cdots\odot x_n$ in $\mathcal U(\overline{\g})$. From this observation follows that $\phi$ restricts to the identity map from $\mathcal U_{\leq 1}(\overline{\g})$ to $\mathcal U_{\leq 1}(\g)$. Moreover, a simple induction on the length of the monomial shows that 
\begin{equation}
\phi(x_1\odot\cdots\odot x_n)-x_1\cdot x_2\cdots x_n\in\mathcal U_{\leq n-1}(\g),\label{eq:aprox}
\end{equation}
for all monomial of degree $n$ in $\mathcal U(\overline{\g})$. From this one easily deduces that the restriction of $\phi$ to $\mathcal U_{\leq n}(\overline{\g})$ surjects onto $\mathcal U_{\leq n}(\g)$, for all $n\geq 2$. On the other hand, if such a restriction had non trivial kernel, then there should be a monomial of length $n$, say $x_1\cdot x_2\cdots x_n$, contained in $\mathcal U_{\leq n-1}(\g)$, which would imply that $x_1\cdot x_2\cdots x_n=0$. Since $\mathcal U(\g)$ is an integral domain, at least one $x_i$ should be equal to zero, which, in turn would imply that $x_1\odot\cdots\odot x_n=0$, proving that, for all $n\geq 2$, the restriction of $\phi$ is a linear isomorphism between $\mathcal U_{\leq n}(\overline{\g})$ and $\mathcal U_{\leq n}(\g)$. We are left to show that $\phi$ is compatible with the coalgebra structures of the universal enveloping algebras.
The compatibility of $\phi$ with the counits is clear. Then note that if $A=x\odot X$ where $X$ is a monomial of length $n-1$
\begin{equation}
\phi(A)=x\cdot\phi(X)+x\triangleright\phi(X).\label{eq:recurs}
\end{equation}
Writing $\Delta$ to denote the coproduct both in $\mathcal U(\overline{\g})$ and in $\mathcal U(\g)$
one has 
\[
\Delta\circ\phi(x)=(\phi\otimes\phi)\circ\Delta(x),\,\forall x\in\overline{\g}.
\]
Suppose that this identity holds for every monomial $X\in\mathcal U(\overline{\g})$ of length $n-1$, i.e.
\begin{equation}
\big(\phi (X)\big)_{(1)}\otimes \big(\phi (X)\big)_{(2)}=\Delta\circ\phi(X)=(\phi\otimes\phi)\circ\Delta (X)=\phi(X_{(1)})\otimes\phi(X_{(2)}),\label{eq:recur2}
\end{equation}
then, if $A=x\odot X$, one can compute 
\begin{eqnarray*}
&&(\phi\otimes\phi)\circ \Delta(A)=(\phi\otimes\phi)\circ \Delta(x\odot X)=(\phi\otimes\phi)(\Delta x\odot\Delta X)\\
&=&(\phi\otimes\phi)\big((x\odot X_{(1)})\otimes X_{(2)}+X_{(1)}\otimes (x\odot X_{(2)})\big)\\
&=&\big(x\cdot\phi(X_{(1)})+x\triangleright\phi(X_{(1)})\big)\otimes\phi(X_{(2)})+\phi(X_{(1)})\otimes \big(x\cdot\phi(X_{(2)})+x\triangleright\phi(X_{(2)})\big).
\end{eqnarray*}
On the other hand
\begin{eqnarray*}
&&\Delta\circ\phi (A)\stackrel{\eqref{eq:recurs}}{=}\Delta\big(x\cdot\phi(X)+x\triangleright\phi(X)\big)=\Delta\big(x\cdot\phi(X)\big)+\Delta \big(x\triangleright\phi(X)\big)\\
&\stackrel{\eqref{D bial item3}}{=}&\Delta x\cdot\Delta\big(\phi(X)\big)+\Delta x\triangleright\Delta (\phi(X))\\ 
&\stackrel{\eqref{eq:recur2}}{=}&\Delta x\cdot (\phi\otimes\phi)(X_{(1)}\otimes X_{(2)})+\Delta x\triangleright(\phi\otimes\phi)(X_{(1)}\otimes X_{(2)})\\
&=&\big(x\cdot\phi(X_{(1)})+x\triangleright\phi(X_{(1)})\big)\otimes\phi(X_{(2)})+\phi(X_{(1)})\otimes \big(x\cdot\phi(X_{(2)})+x\triangleright\phi(X_{(2)})\big).
\end{eqnarray*} 
\end{proof}

Recall the product $\ast$ defined in \eqref{eq:ast prod}. 

\begin{thm}[\cite{KLM}]\label{thm:klm}
$\phi$ is an isomorphism of bialgebras between $(\mathcal U(\overline{\g}),\odot,1,\Delta,\epsilon)$ and $(\mathcal U(\g),\ast,1,\Delta,\epsilon)$. 
\end{thm}
\begin{proof}
 It is a direct verification. 
\end{proof}
\begin{rem}
The previous theorem was first proven in \cite{KLM}, even though there it was phrased in a slightly different way. The analogue if this statement for pre-Lie algebra was first shown to be true in \cite{GO1,GO2}.
\end{rem}

\begin{rem}\label{rem:category}
From a more categorical point of view one can rephrase the result of \ref{thm:klm} in terms of the existence of an \emph{isomorphism of functors} 
\begin{equation}
\phi\co \mathcal{U}\circ \Join \to \ast \circ \mathcal{U}\label{key}
\end{equation} 
extending the identity 
\[
\Prim \circ\, \mathcal{U}\circ \Join =\Prim \circ \ast \circ\, \mathcal{U}.
\]
In the previous formulas, the functor $\ast$ was defined in \eqref{eq:functor ast}, $\mathcal U$, and $\operatorname{Prim}$ were defined in \eqref{eq:univ funct PSB} and respectively in \eqref{eq: functor primitive} while 
\[
\Join \co \cat{PostLie} \to  \cat{Lie}
\]
is the functor induced by the bracket \eqref{eq:newLie}.
\end{rem}

In what follows we will need to work with a suitable completion of the Hopf algebras $\mathcal U_{\cdot}(\g)$, ${\mathcal U}_{\ast}(\g)$ and $\mathcal U_{\odot}(\overline{\g})$. We will denote these completions as $\hat{\mathcal U}_{\cdot}(\g)$, $\hat{{\mathcal U}}_{\ast}(\g)$ and $\hat{\mathcal U}_{\odot}(\overline{\g})$, without making any notational difference between the structural operations of the original and the completed Hopf algebras (i.e. we will denote by $\Delta$ both the coproduct on the original and of the complete Hopf algebra). The completions we will be interested in are obtained as inverse limits of quotients of the original Hopf algebras by the powers of their augmentation ideals, and in these enlarged Hopf algebras it will make sense consider infinite series like exponentials or logarithms. Furthermore, the original Lie algebras $\g$ and $\overline{\g}$ inherit a completion from the ambient completed universal enveloping algebras. To save notation we will not introduce new symbols to distinguish between the original and the completed Lie algebras, hoping that it will be clear from the context which Lie algebras we are considering. For more details about the completion of Hopf algebras we refer the reader to \cite{Quillen}, see also \cite{FloyMunth} and \cite{KI}. 

Note that since $\phi(x_1\odot\cdots\odot x_n)=x_1\ast\cdots\ast x_n$, the map $\phi$ extends to an isomorphism $\phi:\hat{\mathcal U}_{\odot}(\overline{\g})\rightarrow \hat{\mathcal U}_{\ast}(\g)$ of complete Hopf algebras. Moreover, denoting with $\diamond$ any one of the products $\cdot,\ast$ and $\odot$, with $\mathfrak G$ any one of the Lie algebras $\g$ or $\overline{\g}$, and with $\exp_{\diamond}$ and $\log_{\diamond}$ the corresponding \emph{exponential} and \emph{logarithm} maps, one has that
$\xi$ is a \emph{group-like} element of $\hat{\mathcal U}_{\diamond}(\mathfrak G)$ \em{if and only if} $\xi=\exp_{\diamond}(x)$ for a \em{unique} $x\in\operatorname{Prim}\big(\hat{\mathcal U}_{\diamond}(\mathfrak G)\big)$. In other words, 
\[
\exp_{\diamond}:\operatorname{Prim}(\hat{\mathcal U}_{\diamond}(\mathfrak G))\rightarrow\operatorname{Group}(\hat{\mathcal U}_{\diamond}(\mathfrak G)),
\]
is a bijection whose inverse is 
\[
\log_{\diamond}:\operatorname{Group}(\hat{\mathcal U}_{\diamond}(\mathfrak G))\rightarrow \operatorname{Prim}(\hat{\mathcal U}_{\diamond}(\mathfrak G)).
\]

\begin{rem}
Recall that $\operatorname{Prim}(\hat{\mathcal U}_{\diamond}(\mathfrak G))=\mathfrak G$.
\end{rem}

In particular the application $\eta:\overline{\g}\rightarrow\g$, defined by
\begin{equation}
\eta=\log_{\cdot}\circ\phi\circ\exp_{\odot}\label{eq:eta},
\end{equation}
is a bijection. 
\begin{rem} Since $\overline{\g}$ and $\g$ are two Lie algebra having the same underlying vector space, $\eta$ can be thought of as a map between $\g$ and itself. Furthermore, note that $\phi:\hat{\mathcal U}_{\odot}(\overline\g)\rightarrow\hat{\mathcal U}_{\ast}(\g)$ restricts to a bijection $\phi:\operatorname{Group}(\hat{\mathcal U}_{\odot}(\overline\g))\rightarrow\operatorname{Group}(\hat{\mathcal U}(\g))$. In particular, since $\hat{\mathcal U}_{\cdot}(\g)$ and $\hat{\mathcal U}_{\ast}(\g)$ have the same coalgebra structure,
$\phi(\exp_{\odot}x)=\exp_{\ast}x$ is a group-like element of $\hat{\mathcal U}_{\cdot}(\g)$. 
\end{rem}

Let $\chi:\g\rightarrow\g$ be the inverse of $\eta$, i.e. $\chi$ is the application that takes every $x\in\g$ to the (unique) element $\chi(x)\in\g$ such that 
\[
\exp_{\cdot}(x)=\exp_{\ast}\big(\chi(x)\big).
\]
\begin{defn}\label{de: pl magnus expansion}
The map $\chi$ is called the \emph{post-Lie Magnus expansion}.
\end{defn}

\begin{rem}\label{rem:mag}
The map $\chi$ is a very interesting mathematical object. It was introduced in \cite{EFLMMK}, see also \cite{KLM}, to analyze iso-spectral type flow equations defined on a post-Lie algebra whose post-Lie product was coming from a solution of the modified Yang-Baxter equation. More in general, in \cite{KI}, it was observed that given a post-Lie algebra $(\g,\triangleright)$, for every $x\in\g$,  $\chi_x(t):=\chi(tx)\in\g[[t]]$ satisfies the following non-linear ODE
\[
\dot{\chi}_x(t)=(d\exp_{\ast})^{-1}_{-\chi_x(t)}\big(\exp_{\ast}(-\chi_x (t))\triangleright x\big),
\]
and that, the non-linear post-Lie differential equation 
\[
{\dot x}(t)=-x(t)\triangleright x(t),
\]
for $x=x(t)\in\g[[t]]$, with initial condition $x(0)=x_0\in\g$, has as a solution
\[
x(t)=\exp_\ast (-\chi_{x_0}(t))\triangleright x_0.
\]
In the same reference $\chi$ was dubbed \emph{post-Lie Magnus expansion} to stress that such a map is the analogue, for post-Lie algebras, of the so called \emph{pre-Lie Magnus expansion}, see \cite{Manchon} and references therein. This can be defined on every \emph{completed} and \emph{unital} pre-Lie algebra $(\g,\triangleright)$ as the map $\Omega:\g\rightarrow\g$ satisfying the recursive relation
\[
\Omega(x)=\sum_{k=0}^\infty\frac{B_k}{k!}L^k_{\Omega(x)}(x),
\]
where, for all $y\in\g$, $L_y:\g\rightarrow\g$, is defined by $L_y(x)=y\triangleright x$ and the $B_k$ are the Bernoulli numbers, $B_0=1,B_1=-\frac{1}{2}, B_2=\frac{1}{6}, B_3=0,....$. In particular, the first terms of the previous expansion read as
\begin{equation}
\Omega(x)=x-\frac{1}{2}x\triangleright x+\frac{1}{4}(x\triangleright x)\triangleright x+\frac{1}{12}x\triangleright(x\triangleright x)+\cdots\label{eq:preLieMag}
\end{equation}
which is the \emph{compositional} inverse of the (left) \emph{pre-Lie exponential map} 
\[
\exp_{\triangleright}(x):=x+\frac{1}{2}x\triangleright x+\frac{1}{6}x\triangleright (x\triangleright x)+\cdots
\]
In other words, the (left) pre-Lie Magnus expansion is the (left) pre-Lie logarithm, i.e. $\Omega(x)=\log_{\triangleright}(1+x)$, for all $x\in\g$. The pre-Lie Magnus expansion turned out to be an important object in several different areas of mathematics like dynamical systems, see \cite{AG}, combinatorics \cite{CP}, \cite{KM} and \cite{BS}, quantum field theory \cite{EFP} and deformation theory, see \cite{bandiera} and \cite{DSV}. At the best of our knowledge, \eqref{eq:preLieMag} was dubbed pre-Lie Magnus expansion in \cite{KM}. A very nice and comprehensive review of the classical Magnus expansion and of its many applications can be found in \cite{BCOR}, see also \cite{C}.
\end{rem}
Let $(\g,\triangleright)$ be a post-Lie algebra and let $\operatorname{Exp}(\nabla_x)$ be the automorphism of (the vector space) $\g$ defined by 
\begin{equation}
\operatorname{Exp}(\nabla_x)(y)=y+\sum_{n\geq 1}\frac{\nabla^n_x(y)}{n!}\label{eq:canonical aut}. 
\end{equation}

\begin{prop}\label{prop:aut}
 For all $x\in\g$, one has 
\begin{equation} 
\operatorname{Exp}\big(\nabla_{\chi(x)}\big)(y)=\sum_{n\geq 0}\frac{1}{n!}T_n(\underbrace{x,\dots,x}_{n-\text{times}};y),\,\forall y\in\g.\label{eq:bracexp}
\end{equation}
\end{prop}
\begin{proof}
Recall from \eqref{eq:partialfor}, that for $a\in \g$ and $b\in U(\g)$, one has $a\ast b=ab+ a\triangleright b$.
Using \ref{D bial item5}, this gives $a\triangleright( b\triangleright y)=(a\ast b)\triangleright y$. We deduce that   
\begin{equation*}
\operatorname{Exp}\big(\nabla_{\chi(x)}\big)(y)=y+ \chi(x)\triangleright y  + \frac{1}{2!}\big(\chi(x)\ast \chi(x)\big)\triangleright y +  ... = \exp_{\ast}\big(\chi(x)\big)\triangleright y.
\end{equation*} 
The latter is, by definition of $\chi$, equal to $\exp_{\cdot}(x)\triangleright y$. 
\end{proof}
\begin{rem}
It is worth to mention that \eqref{eq:bracexp} was proven in \cite{EFP} for the case of a pre-Lie algebra, see also \cite{DSV} and \cite{bandiera}.
\end{rem}
We now discuss how the post-Lie Magnus expansion relates to the \emph{Hausdorff} groups of the Lie algebras $\overline\g$ and $\g$. Recall that the Hausdorff group of a complete Lie algebra $(\g,[-,-])$ is the group with product defined by the Baker-Campbell-Hausdorff series  $\operatorname{BCH}(x,y)=\log(e^xe^y)$
whose the first terms read as
\[
\operatorname{BCH}(x,y)=x+y+\frac{1}{2}[x,y]+\frac{1}{12}([x,[x,y]]+[y,[y,x]])+\cdots
\]
To achieve our goal, let $\sharp:\g\otimes\g\rightarrow\g$ be defined by
\begin{equation}
x\sharp y=\log_{\cdot}\big(\exp_{\cdot}(x)\ast\exp_{\cdot}(y)\big),\,\forall x,y\in\g. \label{eq:compprod}
\end{equation}
This operation is called \emph{composition product} in \cite{FloyMunth} to which we refer for more informations about its relevance in the theory of geometric numerical integration. 
\begin{lem}[\cite{FloyMunth}, Proposition 2.5, p.12]\label{lem:compro}
	For all $x,y\in\g$, one has 
\begin{equation}
x\sharp y=\operatorname{BCH}_{[-,-]}\big(x,\exp_{\cdot}(x)\triangleright y\big). \label{eq:compbch1}
\end{equation}
\end{lem}
\begin{proof}
For completeness we recall here below the proof of this result. Since $\exp_{\cdot}(x)$ is a group-like element, \eqref{eq:ast prod} implies that 
\[
\exp_{\cdot}(x)\ast\exp_{\cdot}(y)=\exp_{\cdot}(x)\cdot(\exp_{\cdot}(x)\triangleright\exp_{\cdot}(y)).
\]
A simple induction on the length of the monomials, together with \ref{D bial item4}, gives 
\[
\exp_{\cdot}(x)\triangleright \underbrace{y\cdots y}_{n-\text{times}}=\underbrace{\exp_{\cdot}(x)\triangleright y\cdots\exp_{\cdot}(x)\triangleright y}_{n-\text{times}},
\]
which implies
\[
\exp_{\cdot}(x)\ast\exp_{\cdot}(y)=\exp_{\cdot}(x)\exp_{\cdot}(\exp_{\cdot}(x)\triangleright
y)=\exp_{\cdot}(\operatorname{BCH}(x,\exp_{\cdot}(x)\triangleright y)),
\]
proving the statement.
\end{proof}
Going back to the Formula \eqref{eq:compprod} one can compute:
\begin{eqnarray*}
x\sharp y&=&\log_{\cdot}\big(\exp_{\cdot}(x)\ast\exp_{\cdot}(y)\big)\\
&=&\log_{\cdot}\big(\exp_{\ast}(\chi(x))\ast\exp_{\ast}(\chi(y))\big)\\
&=&\log_{\cdot}\big(\phi\big(\exp_{\odot}(x)\big)\ast\phi\big(\exp_{\odot}(y)\big)\big)\\
&=&\log_{\cdot}\big(\phi\big(\exp_{\odot}(x)\odot\exp_{\odot}(y)\big)\big)\\
&=&\log_{\cdot}\big(\phi\big(\exp_{\odot}\big(\operatorname{BCH}_{[[-,-]]}(\chi(x)),\chi(y)\big)\big)\big)\\
&=&\log_{\cdot}\big(\exp_{\ast}\big(\operatorname{BCH}_{[[-,-]]}(\chi(x)),\chi(y)\big)\big).\\
\end{eqnarray*}
Therefore, one has 
\[
\exp_{\cdot}(x\sharp y)=\exp_{\ast}\big(\operatorname{BCH}_{[[-,-]]}(\chi(x)),\chi(y)\big), 
\]
which, using the definition of the map $\chi$, becomes
\[
\exp_{\ast}\big(\chi(x\sharp y)\big)=\exp_{\ast}\big(\operatorname{BCH}_{[[-,-]]}(\chi(x)),\chi(y)\big),
\]
or, equivalently,
\[
\chi(x\sharp y)=\operatorname{BCH}_{[[-,-]]}(\chi(x)),\chi(y)\big).
\]
Using \eqref{eq:compbch1} and the identity 
\[
\exp_{\cdot}(x)\triangleright y=\sum_{n\geq 0}\frac{1}{n!}T_n(\underbrace{x,\dots,x}_{n-\text{times}};y),\,\forall x,y\in\g,
\]
one obtains the following. 
\begin{prop}\label{prop:bchprop}
For all $x,y\in\g$, one has 
\begin{equation}
\operatorname{BCH}_{[[-,-]]}(\chi(x),\chi(y))=\chi\big(\operatorname{BCH}_{[-,-]}(x,\operatorname{Exp}(\nabla_{\chi(x)})(y))\big).\label{eq:compbch2}
\end{equation}
\end{prop}

If $(\g,\triangleright)$ is a pre-Lie algebra, i.e. a post-Lie algebra such that $[-,-]\equiv 0$, then 
\[
\chi(x)=\log_{\triangleright}(1+x),
\]
see Remark \ref{rem:mag}, i.e. under this assumption the post-Lie Magnus expansion \emph{is} the pre-Lie Magnus expansion, see Corollary 8 pag. 276 in \cite{KI}. 
 From Proposition \ref{prop:bchprop} one gets the following. 
\begin{cor}\label{cor:ag}
If $\g$ is a pre-Lie algebra, then, for all $x,y\in\g$, one has 
\begin{equation}
\exp_{\triangleright}\big(\operatorname{BCH}_{[[-,-]]}(\chi(x),\chi(y))\big)=1+x+\operatorname{Exp}(\nabla_{\chi(x)})(y).\label{eq:compbch3}
\end{equation}
\end{cor}
\begin{rem}\label{rem:explic}
The identity \eqref{eq:compbch3} was first proven in \cite{AG}, where it was also observed that the operation $f:\g\otimes\g\rightarrow\g$ defined by 
\begin{equation}
f(x,y)=x+\operatorname{Exp}(\nabla_{\chi(x)})(y),\,\forall x,y\in\g,\label{eq:formalaw}
\end{equation}
turns $\g$ into a group, there termed the group of \emph{formal flows} of $\g$, isomorphic to the Hausdorff group of the Lie algebra underlying the pre-Lie algebra $(\g,\triangleright)$, see also \cite{Manchon, EFP, bandiera} and \cite{DSV}. It is probably worth to stress that in the pre-Lie case the symmetric brace algebra is a crucial ingredient to \emph{integrate} the original pre-Lie algebra, see Formulas \eqref{eq:bracexp} and \eqref{eq:formalaw} while the pre-Lie Magnus expansion is used to define the isomorphism between the group of formal flows and the Baker-Campbell-Hausdorff group of the Lie algebra underlying the original pre-Lie algebra. In analogy with the pre-Lie case, Proposition \ref{prop:bchprop} shows that the post-symmetric braces and the post-Lie Magnus expansion are crucial ingredients to \emph{integrate} a (complete) post-Lie algebra.

It is also worth to remark that Lemma \ref{lem:compro}, together with Definition \ref{de: pl magnus expansion}, tells us that for every (unital and complete) pre-Lie algebra the composition law of the formal flows \eqref{eq:formalaw} is nothing else than the composition product defined in \eqref{eq:compprod}.
\end{rem}

\section*{Acknowledgements}
The second author was supported by PNPD/CAPES-2013 during the first period of this project, and by grant ``\#2018/19603-0, S\~ao Paulo Research Foundation (FAPESP)''  during the second period. 
The third author thanks the Instituto de Ci\^encias Matem\'aticas e de Computa\c c\~ao, where the first part of this project was developed during her doctorate.  

The authors wish to express their gratitude to Kurusch Ebrahimi-Fard and Dominique Manchon for useful comments about the content of the paper.

\bibliographystyle{abbrv}

\end{document}